\documentclass[a4paper]{amsart}
\usepackage{amsmath,amsthm,amssymb,latexsym,epic,bbm,comment,mathrsfs}
\usepackage{graphicx,enumerate,stmaryrd,color,tikz,ytableau,xcolor}
\usepackage[all,2cell]{xy}
\xyoption{2cell}
\usetikzlibrary{patterns,snakes}

\newtheorem{theorem}{Theorem}
\newtheorem{lemma}[theorem]{Lemma}
\newtheorem{corollary}[theorem]{Corollary}
\newtheorem{proposition}[theorem]{Proposition}
\newtheorem{conjecture}[theorem]{Conjecture}
\newtheorem{example}[theorem]{Example}
\newtheorem{definition}[theorem]{Definition}

\usepackage[all]{xy}
\usepackage[active]{srcltx}
\usepackage[parfill]{parskip}
\usepackage{enumerate}

\begin{document}
\title[Kostant cuspidal permutations]
{Kostant cuspidal permutations}

\author[Samuel Creedon and Volodymyr Mazorchuk]
{Samuel Creedon and Volodymyr Mazorchuk}

\begin{abstract}
In relation to Kostant's problem for simple highest weight
modules over the general linear Lie algebra, we prove a
persistence result for Kostant negative consecutive patterns.
Inspired by it, we introduce the notion of a Kostant cuspidal
permutation as a minimal Kostant negative consecutive pattern.
It is shown that Kostant cuspidality is an invariant of a 
Kazhdan-Lusztig left cell. We describe four infinite families 
of Kostant cuspidal involutions, including a complete classification
of Kostant cuspidal fully commutative involutions. In particular,
we show that the number of new Kostant cuspidal elements 
can be arbitrarily large, when the rank grows. This provides
some potential explanation why Kostant's problem is hard.
\end{abstract}

\maketitle

\section{Introduction and description of results}\label{s1}

Kostant's problem is a classical open question about 
Lie algebra representations, see \cite{Jo}. Given a 
finite dimensional Lie algebra $\mathfrak{g}$ and 
a $\mathfrak{g}$-module $M$, {\em Kostant's problem} 
for the module $M$ is the question whether the 
universal enveloping algebra of $\mathfrak{g}$ surjects
onto the algebra of adjointly finite linear endomorphisms of 
$M$. The answer is known for some classes of modules
but is unknown in general, even for very classical 
classes of modules, for example, for simple highest weight
modules over semi-simple complex Lie algebras.
Naturally, we will call $M$ {\em Kostant positive}
provided that the answer to Kostant's problem for
$M$ is positive and {\em Kostant negative} otherwise.

Recently, it was observed (see \cite{KMM}) that Kostant's 
problem  for simple highest weight modules is
closely connected to the action of the bicategory
of projective functors (see \cite{BG}) on the simple
highest weight module in question. This observation was
a biproduct of an attempt to understand the conjecture about
Kostant's problem for some simple highest weight modules 
proposed by Johan K{\aa}hrstr{\"o}m, see \cite{KMM}.
The conjecture relates Kostant's problem to combinatorics
of the Hecke algebra. 

The insight of this conjecture allowed for several advances
in our understanding of Kostant's problem. To start with, 
it led to a complete answer to Kostant's problem for simple
highest weight modules indexed by fully commutative 
permutations, see \cite{MMM}. This answer was explicit enough
to allow for enumeration of positive and negative solutions
which resulted in a conjecture that, for simple
highest weight modules in the regular block of the BGG 
category $\mathcal{O}$ for $\mathfrak{sl}_n$, the
answer is almost always negative, asymptotically, when
$n$ tends to $\infty$.

This conjecture was proved in \cite{CM1}. The proof is 
based on combinatorics of consecutive patterns. The 
idea is as follows: the smallest Kostant negative simple
highest weight module in the regular block of 
category $\mathcal{O}$ for $\mathfrak{sl}_n$ is the module
corresponding to the permutation $2143$ in $S_4$
(here we use the one-line notation) for the Lie algebra $\mathfrak{sl}_4$.
The main observation of \cite{CM1} was that, given 
an element $w\in S_n$, the corresponding simple
highest weight module is Kostant negative provided
that $w$ consecutively contains the pattern $2143$.

The main result of the present paper,
see Theorem~\ref{thm-main}, is the following 
vast generalization of the observation in \cite{CM1}.
\vspace{2mm}

{\bf Theorem A.} {\em Let $w\in S_n$ be such that the
corresponding simple highest weight module $L_w$ 
in the regular block of  category $\mathcal{O}$ for 
$\mathfrak{sl}_n$ is Kostant negative. 
Let $x\in S_k$, for $k\geq n$, be such that 
$x$ consecutively contains $w$. Then the
corresponding simple highest weight module $L_x$ 
in the regular block of  category $\mathcal{O}$ for 
$\mathfrak{sl}_k$ is Kostant negative. }
\vspace{2mm}

Adapting the terminology from \cite{Mat}, we call
$w\in S_n$ {\em Kostant cuspidal} provided that 
it is Kostant negative, however, all proper
consecutive patterns of $w$ are Kostant positive.
Our second main result is the following,
see Theorems~\ref{thm-s7.3-1}, \ref{thm-cusp1}, \ref{thm-cusp2} 
and \ref{thm-cusp3}:

{\bf Theorem B.} {\em 
\begin{enumerate}[$($a$)$]
\item The only fully commutative Kostant cuspidal involutions of 
$S_n$ are the elements 
\begin{displaymath}
(i+1)(i+2)\dots (2i)12\dots i
(j+1)(j+2)\dots n(2i+1)(2i+2)\dots j
\end{displaymath}
(one-line notation),
where $i=1,2,\dots,k-1$ and $j=2i+\frac{n-2i}{2}$, for $n=2k$ even.
\item For $n\geq 5$, the element
$1(n-1)(n-2)\dots 2n$ (one-line notation) of $S_n$ is Kostant cuspidal.
\item For $n\geq 7$, the element
$32145\dots (n-3)n(n-1)(n-2)$ (one-line notation) of $S_n$ is Kostant cuspidal.
\item For $n\geq 6$ and $i\in\{3,4,\dots,n-3\}$, the product
$(1,i+1)(i,n)$ of transpositions in $S_n$ is Kostant cuspidal.
\end{enumerate}
}
\vspace{2mm}

Theorem~B implies, in particular, that there are infinitely many Kostant cuspidal
elements, moreover, the number of Kostant cuspidal
elements in $S_n$ can be arbitrarily large when $n$ grows
(in fact, for $n\geq 6$, this number is at least $n-4$). 
In some way,  this explains why Kostant problem is so
difficult. To solve it completely for simple highest weight
modules over all special linear Lie algebras, one has to determine 
all  Kostant cuspidal permutations in all symmetric groups.

It is known, see \cite{MS1}, that the answer to Kostant's problem
is an invariant of a Kazhdan-Lusztig (KL) left cell. This means that,
if two permutations $x$ and $y$ belong to the same 
KL left cell, then the answer to Kostant's problem
for the corresponding simple highest weight modules $L_x$
and $L_y$ is the same. Each KL left cell in a symmetric group 
contains a unique involution. Consequently, it is enough to 
determine Kostant cuspidal involutions. Via the Robinson-Schensted
map, involutions correspond bijectively to standard Young tableaux.
We thus can speak of {\em Kostant cuspidal standard Young tableaux},
defined in the obvious way.
Consequently, to answer Kostant's problem for simple highest weight modules
in the principal block of category $\mathcal{O}$ for the special
linear Lie algebra, it is enough to determine all Kostant cuspidal
standard Young tableaux.

In this paper we use both combinatorial and representation theoretic 
arguments. The combinatorial part mainly addresses various properties of
the Kazhdan-Lusztig combinatorics of the Hecke algebra.
This includes fairly advanced new results, in particular, 
the AI-inspired formulae for the Kazhdan-Lusztig polynomials
from \cite{BBDVW}. The representation theoretic part 
employs detailed knowledge of functorial actions 
on the regular block of the BGG category $\mathcal{O}$
which was developed  during the last 30 years.
At several places, we also use GAP3 computations
(mostly, with help of the CHEVIE package which allows for fairly quick
computations with Kazhdan-Lusztig and dual Kazhdan-Lusztig
bases of the Hecke algebra).

\subsection*{Acknowledgments} This research is supported
by Vergstiftelsen. The fsecond author is also partially
supported by the Swedish Research Council. 
All computations were done using GAP3 and, in particular,
its CHEVIE package, see \cite{GAP3,GAP3-2}.
We used the Microsoft AI Copilot as a general help when writing 
some of our GAP3 codes.

\section{Preliminaries}\label{s2}

\subsection{Symmetric group}\label{s2.1}

For $n\geq 0$, let $S_n$ be the {\em symmetric group} of all permutations of 
$\underline{n}:=\{1,2,\dots,n\}$. For $i=1,2,\dots,n-1$, we denote by
$s_i$ the {\em elementary transposition} $(i,i+1)$. Then $S_n$ is a Coxeter group
if we choose the set of all elementary transposition to be the set of
{\em simple reflections}.

For $\lambda\vdash n$, let $\mathbf{SYT}_\lambda$ be the set of all standard
Young tableaux of shape $\lambda$. Then the classical Robinson-Schensted
correspondence provides a bijection 
\begin{displaymath}
\mathbf{RS}:S_n\overset{\sim}{\longrightarrow}
\coprod_{\lambda\vdash n}
\mathbf{SYT}_\lambda\times \mathbf{SYT}_\lambda,
\end{displaymath}
written $\mathbf{RS}(w)=(P_w,Q_w)$, for $w\in S_n$.
Here $P_w$ is the {\em insertion tableau} and $Q_w$ is the {\em recording
tableau} in Schensted's algorithm, see \cite{Sch,Sa}. Recall
that $w\in S_n$ is an {\em involution}, that is $w^2=e$,
if and only if $P_w=Q_w$. Therefore, sending $w$ to $P_w$,
induces a bijection between the set $\mathrm{Inv}_n$ of all
involutions in $S_n$ and the set 
\begin{displaymath}
\mathbf{SYT}_n:=\coprod_{\lambda\vdash n}
\mathbf{SYT}_\lambda. 
\end{displaymath}

For $1\leq i<j\leq n$, we denote by $\mathrm{inv}(i,j)$ the element
\begin{displaymath}
12\dots(i-1){\color{violet}j(j-1)\dots (i+1)i}(j+1)\dots n\in S_n, 
\end{displaymath}
given in one-line notation, with the non-trivial part colored
in {\color{violet}violet}. The difference between $\mathrm{inv}(i,j)$
and the identity is that we reverse the order for all elements
between $i$ and $j$ (including them).

\subsection{Hecke algebra}\label{s2.2}

Let $\mathbb{A}:=\mathbb{Z}[v,v^{-1}]$.
Let $\mathbf{H}_n$ be the {\em Hecke algebra} of the Coxeter group
$S_n$ over $\mathbb{A}$. It has the {\em standard basis} $\{H_w\,:\,w\in S_n\}$
as well as the {\em Kazhdan-Lusztig basis} $\{\underline{H}_w\,:\,w\in S_n\}$
and the {\em dual Kazhdan-Lusztig basis} $\{\underline{\hat{H}}_w\,:\,w\in S_n\}$,
see \cite{KL}. We use Soergel's normalization from \cite{So}.
The element $\underline{H}_w$ is uniquely determined by the property that
it is invariant under the bar involution of $\mathbf{H}_n$ and has the form
\begin{equation}\label{eq-eqn125}
\underline{H}_w=H_w+\sum_{x<w}p_{w,x}H_x, 
\end{equation}
where $p_{w,v}\in v\mathbb{Z}[v]$ are the {\em Kazhdan Lusztig polynomials}
and $<$ is the Bruhat order on $S_n$.

Recall the {\em Kazhdan-Lusztig left, right and two-sided preorders} on $S_n$,
see \cite{KL}. The left preorder is defined as follows: for $x,y\in S_n$,
we have $x\geq_L y$ provided that there is $w\in W$ such that 
$\underline{H}_x$ appears with a non-zero coefficient when expressing
$\underline{H}_w\underline{H}_y$ with respect to the Kazhdan-Lusztig basis.
The equivalence classes for $\geq_L$ are called {\em left cells}.
We denote the left equivalence of $x,y\in S_n$ by $x\sim_L y$.
The right order $\geq_R$ and the {\em right cells} 
are defined similarly using
$\underline{H}_y\underline{H}_w$.
We denote the right equivalence of $x,y\in S_n$ by $x\sim_R y$.
The two-sided order $\geq_J$ and the {\em two-sided cells} 
are defined similarly using
$\underline{H}_{w_1}\underline{H}_y\underline{H}_{w_2}$. 
We denote the two-sided equivalence of $x,y\in S_n$ by $x\sim_J y$.
From \cite{KL}, we know that
\begin{itemize}
\item for $x,y\in S_n$, we have $x\sim_L y$ if and only if $Q_x=Q_y$;
\item for $x,y\in S_n$, we have $x\sim_R y$ if and only if $P_x=P_y$;
\item for $x,y\in S_n$, we have $x\sim_J y$ if and only if 
the shape of $Q_x$ coincides with the shape of $Q_y$.
\end{itemize}
In particular, the two-sided cells of $S_n$ are in bijection with 
partitions of $n$. Moreover, both left and right cells of $S_n$
are in bijection with $\mathbf{SYT}_n$. In fact, each left and each right
cell contains a unique involution, the so-called {\em Duflo element}.

From \cite{Ge}, we know that $\leq_J$ is given by the 
opposite of the dominance order on partitions.
To the best of our knowledge, so far there is no known combinatorial
description of $\leq_L$ (equivalently, of $\leq_R$).

For $x,y\in S_n$, we have 
$\underline{\hat{H}}_y\underline{H}_x\neq 0$
if and only if $x^{-1}\leq_L y$, see, for example, \cite[Formula~(1)]{KiM}.

As usual, we denote by $\mu$ the {\em Kazhdan-Lusztig $\mu$-function}
and recall the following formulae, for $w\in S_n$ and a simple reflection $s$
(see \cite[Subsection~2.1]{Ir90} or \cite{KL}):
\begin{equation}\label{eq:kl-mu}
\underline{H}_s\underline{H}_w=
\begin{cases}
(v+v^{-1}) \underline{H}_w, & sw<w;\\
\displaystyle\underline{H}_{sw}+
\sum_{x<w,sx<x}\mu(w,x)\underline{H}_x,
& sw>w. 
\end{cases}
\end{equation}
\begin{equation}\label{eq:kl-mu2}
\underline{H}_w\underline{H}_s=
\begin{cases}
(v+v^{-1}) \underline{H}_w, & ws<w;\\
\displaystyle\underline{H}_{ws}+
\sum_{x<w,xs<x}\mu(w,x)\underline{H}_x,
& ws>w. 
\end{cases}
\end{equation}
\begin{equation}\label{eq:kl-mu3}
\underline{H}_s\underline{\hat{H}}_w=
\begin{cases}
0, & sw>w;\\
\displaystyle
(v+v^{-1})\underline{\hat{H}}_w+
\underline{\hat{H}}_{sw}+
\sum_{x>w,sx>x}\mu(w,x)\underline{\hat{H}}_x,
& sw<w. 
\end{cases}
\end{equation}
\begin{equation}\label{eq:kl-mu4}
\underline{\hat{H}}_w\underline{H}_s=
\begin{cases}
0, & ws>w;\\
\displaystyle(v+v^{-1})\underline{\hat{H}}_w+
\underline{\hat{H}}_{ws}+
\sum_{x>w,xs>x}\mu(w,x)\underline{\hat{H}}_x,
& ws<w. 
\end{cases}
\end{equation}

Another useful property is the following: 

\begin{lemma}\label{lem-disjointsupports}
Let  $x,y\in S_n$ have disjoint supports, that is, 
any simple reflection $s$  that appears in a reduced 
expression of $x$ does not appear in a reduced expression of 
$y$ and vice versa. Then we have 
$\underline{{H}}_{xy}=\underline{{H}}_{x}\underline{{H}}_{y}$.
\end{lemma}

\begin{proof}
Since both $\underline{{H}}_{x}$ and $\underline{{H}}_{y}$
are invariant under the bar involution, 
the element $\underline{{H}}_{x}\underline{{H}}_{y}$ is 
invariant under the bar involution as well. Since $x$ and $y$
have disjoint supports, the multiplication map $(u,v)\mapsto uv$
defines a bijection between the set
\begin{displaymath}
\{u\in S_n\,:\, u\leq x\}\times \{v\in S_n\,:\, v\leq y\} 
\end{displaymath}
and the set $\{a\in S_n\,:\, a\leq xy\}$. 
Hence, using \eqref{eq-eqn125}, we obtain
\begin{displaymath}
\underline{H}_x\underline{H}_y\in
H_xH_y+ \sum_{a<xy}v\mathbb{Z}[v]H_a. 
\end{displaymath}
This implies the claim of the lemma.
\end{proof}

\subsection{Category $\mathcal{O}$}\label{s2.3}

Consider the {\em Lie algebra} $\mathfrak{g}=\mathfrak{sl}_n(\mathbb{C})$.
It has the standard {\em triangular decomposition}
\begin{equation}\label{eq1}
\mathfrak{g}=\mathfrak{n}_-\oplus \mathfrak{h}\oplus \mathfrak{n}_+, 
\end{equation}
where $\mathfrak{h}$ is the {\em Cartan subalgebra} of all diagonal matrices
(with trace $0$), $\mathfrak{n}_+$ is the subalgebra of all upper
triangular matrices and $\mathfrak{n}_-$ is the subalgebra of all lower
triangular matrices. The group $S_n$ is isomorphic to the {\em Weyl group }
of the pair $(\mathfrak{g},\mathfrak{h})$.
Associated to the decomposition \eqref{eq1}, we have the
{\em BGG category $\mathcal{O}$} of $\mathfrak{g}$-modules, see \cite{BGG,Hu}.
This category $\mathcal{O}$ is an abelian length category with enough
projectives.

Consider the {\em principal block} $\mathcal{O}_0$ of $\mathcal{O}$.
It is equivalent to the module category over a finite dimensional
associative {\em Koszul algebra} $A$. Assuming that $A$ is basic, it is
defined uniquely up to isomorphism. Simple modules in $\mathcal{O}_0$
are indexed by the elements of $S_n$. For $w\in S_n$, we denote 
by $L_w$ the {\em simple object} in $\mathcal{O}_0$ given by the simple
highest weight $\mathfrak{g}$-module with highest weight 
$w\cdot 0$, where $\cdot$ is the dot-action of $S_n$ on
$\mathfrak{h}^*$. We denote by $\Delta_w$ the {\em Verma cover} of
$L_w$, by $P_w$ the {\em indecomposable projective} cover of
$L_w$ and by $I_w$ the {\em indecomposable injective} envelope of
$L_w$.

Being Koszul, the algebra $A$ is {\em positively graded} and we can
consider the corresponding category of finite dimensional 
graded $A$-modules, denoted ${}^\mathbb{Z}\mathcal{O}_0$. This is
a {\em $\mathbb{Z}$-graded lift} of $\mathcal{O}_0$. All structural
modules admit graded lifts and we denote the standard graded lifts
by the same symbols as the corresponding ungraded modules.

The category $\mathcal{O}_0$ is equipped with an action of the 
monoidal category $\mathscr{P}$ of {\em projective endofunctors}, see \cite{BG}.
The category $\mathscr{P}$ is additive and Krull-Schmidt.
Indecomposable objects in $\mathscr{P}$ are $\theta_w$, 
where $w\in S_n$, they satisfy $\theta_w P_e\cong P_w$.
The category $\mathscr{P}$ admits a graded lift
${}^\mathbb{Z}\mathscr{P}$ which acts on 
${}^\mathbb{Z}\mathcal{O}_0$ such that $\theta_w P_e\cong P_w$,
for $w\in S_n$. 

The Grothendieck group $\mathbf{Gr}[{}^\mathbb{Z}\mathcal{O}_0]$
is isomorphic to $\mathbf{H}_n$ by sending $[\Delta_w]$
to $H_w$, for $w\in S_n$. 
The split Grothendieck ring $\mathbf{Gr}_\oplus[{}^\mathbb{Z}\mathscr{P}]$
is isomorphic to $\mathbf{H}_n^{\mathrm{op}}$ by sending $[\theta_w]$
to $\underline{H}_w$, for $w\in S_n$. In both cases,
multiplication with $v$ corresponds to the shift of 
grading (more specifically, it decreases the grading by $1$).
Via these isomorphisms, the action of ${}^\mathbb{Z}\mathscr{P}$  on 
${}^\mathbb{Z}\mathcal{O}_0$ decategorifies to 
the right regular action of $\mathbf{H}_n$.

\subsection{Kostant's problem}\label{s2.4}

Consider the universal enveloping algebra $U(\mathfrak{g})$.
Given a $\mathfrak{g}$-module $M$, the representation map
embeds $U(\mathfrak{g})/\mathrm{Ann}_{U(\mathfrak{g})}(M)$
into the algebra $\mathrm{End}_\mathbb{C}(M)$ of all linear
endomorphisms of $M$. Let $\mathcal{L}(M,M)$ denote the
$U(\mathfrak{g})$-$U(\mathfrak{g})$-subbimodule of 
$\mathrm{End}_\mathbb{C}(M)$ consisting of all elements,
the adjoint action of $U(\mathfrak{g})$ on which is locally
finite. Then the image of 
$U(\mathfrak{g})/\mathrm{Ann}_{U(\mathfrak{g})}(M)$
belongs to $\mathcal{L}(M,M)$. The classical 
{\em Kostant's problem} for $M$, see \cite{Jo},
is the question whether the embedding
\begin{displaymath}
U(\mathfrak{g})/\mathrm{Ann}_{U(\mathfrak{g})}(M)
\hookrightarrow \mathcal{L}(M,M)
\end{displaymath}
is, in fact, an isomorphism. This problem is open, in general,
even for simple highest weight modules. However, many 
special cases are settled, see \cite{Jo,GJ,Ma25,Ma05,CM2,MMM,KMM,MSr,Ma2,MS,MS1}
and references therein. Notably, in \cite{CM1} it was shown that, 
asymptotically, the answer to Kostant's problem is negative for almost all 
simple modules in $\mathcal{O}_0$.

For a more efficient bookkeeping of the information on Kostant's problem,
let us  introduce the following notation. Denote by $\Sigma$ the set
\begin{displaymath}
\coprod_{n\geq 0}S_n 
\end{displaymath}
and consider the function $\mathbf{K}:\Sigma\to \{+,-\}$, defined as follows: 
for $w\in S_n$, we have $\mathbf{K}(w)=+$ if and only if the answer
to Kostant's problem for the $\mathfrak{sl}_n$-module $L_w$ 
is positive. To emphasize that $w\in S_n$, we will often write
$\mathbf{K}_n(w)$ instead of $\mathbf{K}(w)$.
In \cite{MS1} it was shown that the function $\mathbf{K}$
is constant on left Kazhdan-Lusztig cells. In particular, it is uniquely
determined by its values on all involutions. Therefore we may consider the set
\begin{displaymath}
\mathrm{Inv}=\coprod_{n\geq 0}\mathrm{Inv}_n 
\end{displaymath}
and the restriction of $\mathbf{K}$ to it. As already mentioned above, 
$\mathbf{RS}$ identifies $\mathbf{SYT}_n$ with $\mathrm{Inv}_n$.
Hence we will also consider 
\begin{displaymath}
\mathbf{SYT}=\coprod_{n\geq 0}\mathbf{SYT}_n 
\end{displaymath}
and transfer the function $\mathbf{K}$ from $\mathrm{Inv}$
to $\mathbf{SYT}$ via the identification provided by $\mathbf{RS}$.

\subsection{K{\aa}hrstr{\"o}m's conjecture}\label{s2.5}

As mentioned in \cite[Conjecture~1.2]{KMM}, Johan 
K{\aa}hrstr{\"o}m proposed the following conjecture:

\begin{conjecture}[K{\aa}hrstr{\"o}m]\label{Kh-conj}
For $w\in \mathrm{Inv}_n$, the following conditions are equivalent:
\begin{enumerate}[$($I$)$]
\item\label{kh1} $\mathbf{K}_n(w)=-$.
\item\label{kh2} There exist different $x,y\in S_n$
such that $\theta_xL_w\cong \theta_yL_w\neq 0$
in $\mathcal{O}_0$.
\item\label{kh3} There exist different $x,y\in S_n$
such that $\theta_xL_w\cong \theta_yL_w\neq 0$
in ${}^{\mathbb{Z}}\mathcal{O}_0$.
\item\label{kh4} There exist different $x,y\in S_n$
such that $\underline{\hat{H}}_w\underline{H}_x\vert_{{}_{v=1}}=
\underline{\hat{H}}_w\underline{H}_x\vert_{{}_{v=1}}\neq 0$.
\item\label{kh5} There exist different $x,y\in S_n$
such that $\underline{\hat{H}}_w\underline{H}_x=
\underline{\hat{H}}_w\underline{H}_x\neq 0$.
\end{enumerate}
\end{conjecture}
For $i\in\{II,III,IV;V\}$, we denote by $\mathbf{Kh}^{(i)}$
the map from $\mathrm{Inv}$ (or $\mathbf{SYT}$) 
to $\{+,-\}$ such that 
$\mathbf{Kh}^{(i)}(w)=-$ if and only if the corresponding
$i$-th condition in Conjecture~\ref{Kh-conj} is satisfied.
Then Conjecture~\ref{Kh-conj} amounts to proving that
\begin{displaymath}
\mathbf{K}=\mathbf{Kh}^{(II)}=
\mathbf{Kh}^{(III)}=\mathbf{Kh}^{(IV)}=\mathbf{Kh}^{(V)}.
\end{displaymath}
We refer to \cite{CM3} for various discussions and results 
comparing the conditions appearing in Conjecture~\ref{Kh-conj}.

\subsection{Kostant's problem and homomorphisms}\label{s2.6}

Let $Z(\mathfrak{g})$ be the center of $U(\mathfrak{g})$.
Let $\mathbf{m}$ be the maximal ideal of $Z(\mathfrak{g})$
defined as the kernel of the action of $Z(\mathfrak{g})$
on the trivial $\mathfrak{g}$-module $L_e$.

Let $\mathcal{H}$ denote the category of {\em Harish Chandra}
$U(\mathfrak{g})$-$U(\mathfrak{g})$-bimodules, that is,
finitely generated bimodules, the adjoint action of 
$\mathfrak{g}$ on which is locally finite and has finite
multiplicities, see \cite[Kapitel~6]{Ja}. Consider the
full subcategory ${}_{\,\,0}^\infty\mathcal{H}_0^1$
of $\mathcal{H}$ consisting of all object, the right action of
$\mathbf{m}$ on which is zero and the left action of $\mathbf{m}$
on which is locally nilpotent. The category 
${}_{\,\,0}^\infty\mathcal{H}_0^1$
is equivalent to $\mathcal{O}_0$, see \cite{BG,Ja}. 
Consequently, every indecomposable projective
functor $\theta_w$ can be interpreted as an indecomposable projective 
object in ${}_{\,\,0}^\infty\mathcal{H}_0^1$,
see \cite{BG,Ja}. We will need the following alternative
description of Kostant positivity:

\begin{proposition}\label{prop.K+alt}
For  $x\in S_n$, the module $L_x$ is Kostant positive if and only if,
for any projective functor $\theta$, any homomorphism from
$\theta L_x$ to $L_x$ is an evaluation at $L_x$ of some
natural transformation from $\theta$ to the identity functor.
\end{proposition}

\begin{proof}
For any $x\in S_n$,
we have  $\mathcal{L}(L_x,L_x)\in {}_{\,\,0}^\infty\mathcal{H}_0^1$.
Moreover, for any $w,x\in S_n$, we have an adjunction isomorphism
\begin{equation}\label{eq:homform}
\mathrm{Hom}_{\mathfrak{g}\text{-}\mathfrak{g}}
(\theta_w,\mathcal{L}(L_x,L_x))
\cong
\mathrm{Hom}_{\mathfrak{g}}(\theta_w L_x,L_x).
\end{equation}
For $w=e$, the right hand side of \eqref{eq:homform} is manifestly one-dimensional.
Let $N_x$ denote the image of the unique (up to a non-zero scalar)
non-zero homomorphism from $\theta_e$ to $\mathcal{L}(L_x,L_x)$.
As the identity projective functor $\theta_e$ is realized via the bimodule 
$U(\mathfrak{g})/U(\mathfrak{g})\mathbf{m}$, it follows that 
the module $L_x$ is Kostant positive if and only if 
$N_x=\mathcal{L}(L_x,L_x)$. 

Note that the left hand side of \eqref{eq:homform} has the natural 
module structure over the endomorphism algebra of projective functors.
The endomorphism algebra of projective functors is related to the 
associative algebra underlying $\mathcal{O}_0$ via the above
mentioned equivalence between ${}_{\,\,0}^\infty\mathcal{H}_0^1$
and $\mathcal{O}_0$. In these terms, on the left hand side
of \eqref{eq:homform}, Kostant positivity of $L_x$
means that the image of any homomorphism from a projective
functor $\theta$ to $\mathcal{L}(L_x,L_x)$
belongs to $N_x$, that is, to the sub-object defined as the 
image of the identity functor. By projectivity, this can be 
lifted to a natural transformation $\zeta:\theta\to\theta_e$
(we note that this $\zeta$ is not unique, but it is unique up to
natural transformations that are sent to $0$ by evaluation at $L_x$).
Transferring this information
to the right hand side via the isomorphism \eqref{eq:homform},
we obtain that $L_x$ is Kostant positive if and only if,
for any projective functor $\theta$, every homomorphism from
$\theta L_x$ to $L_x=\theta_e L_x$ is an evaluation at $L_x$
of some such $\zeta$.
\end{proof}

\section{Propagation of Kostant negative consecutive patterns}\label{s3}

\subsection{Main result}\label{s3.1}

\begin{theorem}\label{thm-main}
Let $w\in S_n$ be a Kostant negative element.
For $k\geq n$, let $x\in S_k$ be some element that 
consecutively contains the pattern $w$. Then 
$x$ is Kostant negative.
\end{theorem}

As an immediate corollary of Theorem~\ref{thm-main}, we have:

\begin{corollary}\label{cor-s3.2}
Let $w\in S_n$ be Kostant negative.
For $k\geq n$ and $i\leq k-n$, consider the
embedding $\varphi_{k,n,i}$ of $S_n$ to $S_k$
which sends the transposition $(s,s+1)\in S_n$
to $(s+i,s+i+1)\in S_k$.
Then $\varphi_{k,n,i}(w)\in S_k$ is Kostant negative.
\end{corollary}

We note that the assertion of Corollary~\ref{cor-s3.2}
is not trivial as the natural embedding of the 
Hecke algebra for $S_n$ into the Hecke algebra for
$S_k$ corresponding to $\varphi_{k,n,i}$ does not 
send the elements of the dual Kazhdan-Lusztig basis
for $\mathbf{H}_n$ to the elements of the dual 
Kazhdan-Lusztig basis for $\mathbf{H}_k$.
In fact, we suspected for a long time that 
the assertion of Corollary~\ref{cor-s3.2} holds,
but it is only now that we figured out how to prove it.

\subsection{Proof of Theorem~\ref{thm-main}}\label{s3.2}

Let $w\in S_n$ be such that $\mathbf{K}_n(w)=-$.
For $k\geq n$, let $x\in S_k$ be some element that 
consecutively contains the pattern $w$, starting from 
position $i$. Assume, towards a contradiction, that 
$\mathbf{K}_k(x)=+$. 

Consider the embedding $\varphi_{k,n,i}$ of $S_n$ to $S_k$
which sends the transposition $(s,s+1)\in S_n$
to $(s+i,s+i+1)\in S_k$. Write $x=y\varphi_{k,n,i}(w)$,
such that $\ell(x)=\ell(w)+\ell(y)$. Then 
$y$ is a shortest representative in its coset
from $S_k/\varphi_{k,n,i}(S_n)$.

Let $\mathcal{X}$ be the Serre subcategory of 
$\mathcal{O}_0$ for $\mathfrak{sl}_k$ generated
by all simples $L_r$, where $r\geq y$
(with respect to the Bruhat order).
Let $\mathcal{Y}$ be the Serre subcategory of 
$\mathcal{X}$ generated by all simples $L_r$, 
for which $r\not\in y\varphi_{k,n,i}(S_n)$.

By \cite[Theorem~37]{CMZ}, there is an equivalence between
$\mathcal{O}_0$ for $\mathfrak{sl}_n$
and the Serre quotient $\mathcal{X}/\mathcal{Y}$
which maps $L_z$, for $z\in S_n$, to
$L_{y\varphi_{k,n,i}(z)}$ and, moreover,
which intertwines the action of $\theta_f$, for
$f\in S_n$, on $\mathcal{O}_0$ for $\mathfrak{sl}_n$
with the action of $\theta_{\varphi_{k,n,i}(f)}$
on $\mathcal{X}/\mathcal{Y}$.

Since we assume $\mathbf{K}_k(x)=+$, 
by Proposition~\ref{prop.K+alt}, we have that,
for any projective functor $\theta$, every homomorphism from
$\theta L_x$ to $L_x=\theta_e L_x$ is an evaluation at $L_x$
of some natural transformation from $\theta$ to the identity $\theta_e$.
In particular, this is the case for all $\theta_{\varphi_{k,n,i}(f)}$, where
$f\in S_n$. Applying Proposition~\ref{prop.K+alt} once again,
we deduce that the image of $L_x$ in $\mathcal{X}/\mathcal{Y}$
is Kostant positive. Under the equivalence with 
$\mathcal{O}_0$ for $\mathfrak{sl}_n$, this image is exactly
$L_w$, which gives us our contradiction.

\section{Kostant cuspidal permutations}\label{s4}

\subsection{Definition and first examples}\label{s4.1}

Theorem~\ref{thm-main} motivates the following definition
(our terminology is inspired by \cite{Mat}).

\begin{definition}\label{def-s4.1-1}
An element $w\in S_n$ will be  called {\em Kostant cuspidal}
provided that we have $\mathbf{K}_n(w)=-$ and, for any $k<n$
and  any $p\in S_k$, if $w$ consecutively contains
$p$, then $\mathbf{K}_k(p)=+$.
\end{definition}

\begin{example}\label{ex-s4.1-2}
{\rm
The smallest Kostant negative permutations
$2143$ and $3142$ in $S_4$
(see \cite{MS,KaM}) are, naturally, Kostant cuspidal.
}
\end{example}

\subsection{Left cell invariance}\label{s4.15}

\begin{proposition}\label{ex-s4.15-1}
Let $x,y\in S_n$ be such that $x\sim_L y$.
Then $x$ is Kostant cuspidal if and only if $y$ is.
\end{proposition}

\begin{proof}
Let $x,y\in S_n$ be such that $x\sim_L y$.
Assume that $x$ is Kostant negative, then 
$y$ is Kostant negative as well, by \cite[Theorem~61]{MS1}.
Now let us assume that $x$ is not Kostant 
cuspidal. This means that there is $k<n$
and a Kostant negative $w\in S_k$ such that 
$x$ consecutively contains the pattern $w$
starting from some position, say $i$.

Let $u$ be the consecutive pattern in $y$ of length 
$k$ starting from the same position $i$.
Since $x\sim_L y$, it follows that 
$w\sim_L u$ by \cite[Theorem~1]{Ge2}.
Since $w$ is Kostant negative, we deduce that
$u$ is Kostant negative by \cite[Theorem~61]{MS1}.
Therefore $y$ is not Kostant 
cuspidal. Due to the symmetry in $x$ and $y$,
we obtain that $x$ is Kostant cuspidal if and only if $y$ is,
as asserted.
\end{proof}

As a direct consequence of Proposition~\ref{ex-s4.15-1},
to understand all Kostant cuspidal permutations, it is 
enough to understand all Kostant cuspidal involutions
(equivalently, all Kostant cuspidal standard Young tableaux).

\subsection{Small ranks}\label{s4.17}

As already mentioned in Example~\ref{ex-s4.1-2},
the first Kostant negative, and hence also Kostant cuspidal, 
involution is $2143\in S_4$, see \cite{KaM}. The
corresponding standard Young tableau is
\begin{displaymath}
\begin{ytableau}
\scriptstyle{1}&
\scriptstyle{3}\\
\scriptstyle{2}&
\scriptstyle{4}
\end{ytableau} 
\end{displaymath}
In $S_5$, there are exactly five Kostant negative involutions,
namely $21435$, $13254$, $21543$, $32154$ and $14325$
(see \cite{KaM}).
We see that the first four consecutively contain the pattern
$2143$ and hence they are not Kostant cuspidal. The last 
involution $14325$ does not contain any proper Kostant negative
consecutive pattern and hence is Kostant cuspidal. The
corresponding standard Young tableau is
\begin{displaymath}
\begin{ytableau}
\scriptstyle{1}&
\scriptstyle{2}&
\scriptstyle{5}\\
\scriptstyle{3}\\
\scriptstyle{4}
\end{ytableau} 
\end{displaymath}
In $S_6$, there are exactly twenty five Kostant negative involutions
(see \cite{KMM}), namely:
\begin{equation}\label{eqeqnnw2}
\begin{array}{ccccc}
124365,\quad &  % 35  contains 2143
132546,\quad &  % 24  contains 2143
214356,\quad &  % 13  contains 2143
125436,\quad &  % 343 contains 14325
143256,\\       % 232 contains 14325
214365,\quad &  % 135 contains 2143
132654,\quad &  % 2454 contains 2143 
143265,\quad &  % 2325 contains 2143
215436,\quad &  % 1343 contains 2143
321546\\        % 1214 contains 2143
{\color{blue}341265},\quad &  % 21325 cuspidal!!!
{\color{blue}215634},\quad &  % 14354 cuspidal!!!
{\color{blue}154326},\quad &  % 232432 cuspidal!!!
216453,\quad &  % 134543 contains 2143
321654\\        % 121454 contains 2143
423165,\quad &  % 123215 contains 2143
351624,\quad &  % 2143254 contains 41523 same left as 21435
216543,\quad &  % 1343543  contains 2143
432165,\quad &  % 1213215  contains 2143
164352\\        % 23245432 cuspidal!!!
{\color{blue}426153},\quad &  % 13214543  cuspidal!!!
524316,\quad &  % 12324321 contains 24315 same as left 14325
463152,\quad &  % 2132145432 contains 3142
526413,\quad &  % 1324321543 contains 3142
632541          % 12134354321 contains 2143
\end{array}
\end{equation}
Kostant cuspidal involutions are highlighted in {\color{blue}blue}.
The corresponding standard Young tableaux are:
\begin{displaymath}
\begin{ytableau}
\scriptstyle{1}&
\scriptstyle{2}&
\scriptstyle{5}\\
\scriptstyle{3}&
\scriptstyle{4}&
\scriptstyle{6}
\end{ytableau}, \quad 
\begin{ytableau}
\scriptstyle{1}&
\scriptstyle{3}&
\scriptstyle{4}\\
\scriptstyle{2}&
\scriptstyle{5}&
\scriptstyle{6}
\end{ytableau}, \quad 
\begin{ytableau}
\scriptstyle{1}&
\scriptstyle{2}&
\scriptstyle{6}\\
\scriptstyle{3}\\
\scriptstyle{4}\\
\scriptstyle{5}
\end{ytableau}, \quad 
\begin{ytableau}
\scriptstyle{1}&
\scriptstyle{3}\\
\scriptstyle{2}&
\scriptstyle{5}\\
\scriptstyle{4}&
\scriptstyle{6}
\end{ytableau}
\end{displaymath}
The element $154326$ belongs to the infinite series 
discussed later in Theorem~\ref{thm-cusp1}.
The element $426153$ belongs to the infinite series 
discussed later in Theorem~\ref{thm-cusp3}.

\subsection{Type $A_6$}\label{s4.55}

The last case when we know a complete answer to Kostant's
problem is type $A_6$, that is in $S_7$, see \cite{CM1}. In this case there are
exactly $107$ Kostant negative involutions. Based on this
explicit result,
a GAP3 computation produced the following list of nine
cuspidal involutions in type $A_6$ (using the one-line notation):
\begin{displaymath}
1462537,\quad
1536247,\quad
{\color{blue}3214765},\quad
4531276,\quad
3614725,
\end{displaymath}
\begin{displaymath}
2167534,\quad
{\color{violet}1654327},\quad
{\color{teal}4271563},\quad
{\color{teal}5237164}.
\end{displaymath}
The corresponding 
standard Young tableaux are:
\begin{displaymath}
\begin{ytableau}
\scriptstyle{1}&
\scriptstyle{2}&
\scriptstyle{3}&
\scriptstyle{7}\\
\scriptstyle{4}&
\scriptstyle{5}\\
\scriptstyle{6}
\end{ytableau}, \quad 
\begin{ytableau}
\scriptstyle{1}&
\scriptstyle{2}&
\scriptstyle{4}&
\scriptstyle{7}\\
\scriptstyle{3}&
\scriptstyle{6}\\
\scriptstyle{5}
\end{ytableau}, \quad 
\begin{ytableau}
\scriptstyle{1}&
\scriptstyle{4}&
\scriptstyle{5}\\
\scriptstyle{2}&
\scriptstyle{6}\\
\scriptstyle{3}&
\scriptstyle{7}
\end{ytableau}, \quad 
\begin{ytableau}
\scriptstyle{1}&
\scriptstyle{2}&
\scriptstyle{6}\\
\scriptstyle{3}&
\scriptstyle{5}&
\scriptstyle{7}\\
\scriptstyle{4}
\end{ytableau},\quad 
\begin{ytableau}
\scriptstyle{1}&
\scriptstyle{2}&
\scriptstyle{5}\\
\scriptstyle{3}&
\scriptstyle{4}&
\scriptstyle{7}\\
\scriptstyle{6}
\end{ytableau}, 
\end{displaymath}
\begin{displaymath}
\begin{ytableau}
\scriptstyle{1}&
\scriptstyle{3}&
\scriptstyle{4}\\
\scriptstyle{2}&
\scriptstyle{5}&
\scriptstyle{7}\\
\scriptstyle{6}
\end{ytableau}, \quad 
\begin{ytableau}
\scriptstyle{1}&
\scriptstyle{2}&
\scriptstyle{7}\\
\scriptstyle{3}\\
\scriptstyle{4}\\
\scriptstyle{5}\\
\scriptstyle{6}
\end{ytableau}, \quad 
\begin{ytableau}
\scriptstyle{1}&
\scriptstyle{3}&
\scriptstyle{6}\\
\scriptstyle{2}&
\scriptstyle{5}\\
\scriptstyle{4}&
\scriptstyle{7}
\end{ytableau}, \quad 
\begin{ytableau}
\scriptstyle{1}&
\scriptstyle{3}&
\scriptstyle{4}\\
\scriptstyle{2}&
\scriptstyle{6}\\
\scriptstyle{5}&
\scriptstyle{7}
\end{ytableau}. 
\end{displaymath}
Here 
the {\color{violet}violet} element belongs to the infinite series 
discussed later in Theorem~\ref{thm-cusp1}, the
{\color{blue}blue} element belongs to the infinite series 
discussed later in Theorem~\ref{thm-cusp2} and the two
{\color{teal}teal} elements belong to the infinite series 
discussed later in Theorem~\ref{thm-cusp3}.

\section{Kostant cuspidal fully commutative involutions}\label{s7}

\subsection{Fully commutative permutations}\label{s7.1}

Recall that an element $w\in S_n$ is called {\em fully commutative}
provided that any reduced expression of $w$ can be transformed into
any other reduced expression of $w$ by using only the commutativity
relations between simple reflections. Alternatively, $w$
is fully commutative if it is {\em $321$-avoiding} or
{\em short-braid avoiding}. In terms of the Robinson-Schensted 
correspondence, $w$ is fully commutative if and only if 
the shape of $P_w$ has at most two rows.

For $i\in\{1,2,\dots,n\}$ and $j\in\{0,1,\dots, \min(i-1,n-i-1)\}$,
denote by $\sigma_{i,j}$ the element
\begin{displaymath}
s_i(s_{i+1}s_{i-1})(s_{i+2}s_is_{i-2})\dots
(s_{i-j}s_{i-j+2}\dots s_{i+j})\dots  (s_{i+1}s_{i-1})s_i.
\end{displaymath}
Note that $\sigma_{i,0}=s_i$. The element $\sigma_{i,j}$
is an involution and it can be written as a product of 
$j+1$ transpositions as follows:
\begin{displaymath}
\sigma_{i,j}=(i-j,i+1)(i-j+1,i+2)\dots (i,i+j+1). 
\end{displaymath}
Following \cite{MMM}, we call $\sigma_{i,j}$ a {\em special involution}.
The set of points which $\sigma_{i,j}$ displaces is 
$\{i-j,i-j+1,\dots, i+j+1\}$ and it is called the {\em support} of
$\sigma_{i,j}$. The set $\{i-j-1,i-j,i-j+1,\dots, i+j+2\}$ is called
the {\em extended support} of $\sigma_{i,j}$. Two elements 
$\sigma_{i,j}$ and $\sigma_{i',j'}$ are called {\em distant}
provided that their extended support intersect in at most one element.

\subsection{Kostant positive fully commutative involutions}\label{s7.2}

The following is the main result of \cite{MMM} (see \cite[Theorem~5.1]{MMM}):

\begin{theorem}\label{thm-s7.2-1}
A fully commutative involution is Kostant positive if and only if
it is a product of pairwise-distant special involutions.
\end{theorem}

\subsection{Cuspidal fully commutative involutions}\label{s7.3}

Let $n$ be even, say $n=2k$, with $k>1$. For $a\in\{1,2,\dots,k-1\}$, 
consider the element 
\begin{displaymath}
\tau_{k,a}:=\sigma_{a,a-1}\sigma_{k+a,k-a} \in S_n.
\end{displaymath}
Here are some examples in small ranks (using the one-line notation):
\begin{displaymath}
\tau_{2,1}=2143;\quad
\tau_{3,1}=215634;\quad
\tau_{3,2}=341265;
\end{displaymath}
\begin{displaymath}
\tau_{4,1}=21678345;\quad
\tau_{4,2}=34127856;\quad
\tau_{4,3}=45612387.
\end{displaymath}
Note that both $\tau_{3,1}$ and $\tau_{3,2}$ appear in \eqref{eqeqnnw2}.

Our main result in this section is the following:

\begin{theorem}\label{thm-s7.3-1}
For $n\in\mathbb{Z}_{>0}$, we have:
\begin{enumerate}[$($a$)$]
\item\label{thm-s7.3-1.1} $S_n$ contains a Kostant cuspidal
fully commutative involution if and only if $n\geq 4$ is even. 
\item\label{thm-s7.3-1.2} For an even $n\geq 4$, with $n=2k$,
the Kostant cuspidal
fully commutative involutions in $S_n$ are exactly the elements
$\tau_{k,a}$, where $a\in\{1,2,\dots,k-1\}$.
\end{enumerate}
\end{theorem}

\subsection{Proof of Theorem~\ref{thm-s7.3-1}}\label{s7.4}

Consider $n\geq 4$ with $n=2k$. For 
$a\in\{1,2,\dots,k-1\}$, consider
$\tau_{k,a}:=\sigma_{a,a-1}\sigma_{k+a,k-a}$.
The support of $\sigma_{a,a-1}$ is $\{1,2,\dots,2a\}$
while the support of $\sigma_{k+a,k-a}$ is
$\{2a+1,2a+2,\dots,2k\}$. The two supports are disjoint
but the elements $\sigma_{a,a-1}$ and $\sigma_{k+a,k-a}$ are
not distant. Therefore each $\tau_{k,a}$ as above is
Kostant negative by Theorem~\ref{thm-s7.2-1}.

To prove Kostant cuspidality of $\tau_{k,a}$, we only need to
look at the two extreme consecutive patterns of length $n-1$
in $\tau_{k,a}$ obtained by removing either the first or the last
element in the one-line notation
(i.e. either $\tau_{k,a}(1)$ or $\tau_{k,a}(n)$). Up to the symmetry of the root
system, we only need to consider one case, let us say, the one
when we remove $\tau_{k,a}(n)$. To start with, removing $\tau_{k,a}(n)$ from 
the one-line notation for $\tau_{k,k-1}$
gives the pattern $\sigma_{a,a-1}\in S_{n-1}$ which is Kostant 
positive by Theorem~\ref{thm-s7.2-1}.

To understand what removing of $\tau_{k,a}(n)$ does to $\tau_{k,a}$ with $a<k-1$,
we note that this does not affect the first factor $\sigma_{a,a-1}$
at all. So, we only need to understand what kind of pattern is obtained
if one removes $\sigma_{r,r-1}(2r)$ from $\sigma_{r,r-1}\in S_r$. In the one-line notation,
we first get the element 
\begin{displaymath}
(r+1)(r+2)\dots (2r)12\dots (r-1) 
\end{displaymath}
which corresponds to 
\begin{displaymath}
r(r+1)\dots (2r-1)12\dots (r-1)\in S_{2r-1}. 
\end{displaymath}
The recording tableau of this element is 
\begin{displaymath}
\begin{ytableau}
\scriptstyle{1}&
\scriptstyle{2}&
\scriptstyle{\dots}&
\resizebox{4mm}{!}{$\scriptstyle{r-1}$}&
\scriptstyle{r}\\
\resizebox{4mm}{!}{$\scriptstyle{r+1}$}&
\resizebox{4mm}{!}{$\scriptstyle{r+2}$}&
\scriptstyle{\dots}&
\resizebox{4mm}{!}{$\scriptstyle{2r-1}$}
\end{ytableau} 
\end{displaymath}
This coincides with the recording tableau for
$\sigma_{r,r-2}\in S_{2r-1}$ an hence the two elements
belong to the same left cell. Note that $\sigma_{r,r-2}$
fixes $1$ and hence its support is $\{2,3,\dots,2r-1\}$.

Now let us go back to $\tau_{k,a}$. From the previous paragraph
it follows that the pattern obtained from $\tau_{k,a}$ by removing
$\tau_{k,a}(n)$ is $\sigma_{a,a-1}\sigma_{k+a,k-a-1}$.
The support of $\sigma_{a,a-1}$ is $\{1,2,\dots,2a\}$
while the support of $\sigma_{k+a,k-a-1}$ is
$\{2a+2,2a+3,\dots,2k-1\}$ (note that $\sigma_{k+a,k-a-1}$ fixes $2a+1$) 
and we see that the two factors
are now distant. Therefore $\sigma_{a,a-1}\sigma_{k+a,k-a-1}$
is Kostant positive by Theorem~\ref{thm-s7.2-1}.

The above establishes the fact that all $\tau_{k,a}$ are Kostant cuspidal.
It remains to show that no other fully commutative involutions are
Kostant cuspidal.

To this end, we make the following observation.
Let $w\in S_n$. Recall that the number of rows in $P_w$
is the length of a longest decreasing subsequence in $w$,
see \cite{Sch}. Hence $w\in S_n$ is fully commutative
if and only if the length of a longest decreasing subsequence
in $w$ is either $1$ or $2$. This property is obviously
inherited by any patterns (in particular, by consecutive
patterns) of $w$. Therefore, any consecutive pattern appearing
in a fully commutative element of $S_n$ is itself fully commutative.

Now let us recall the connection between the fully commutative
permutations and Temperley-Lieb diagrams, see \cite{TL}. A 
{\em  Temperley-Lieb diagram} on $2n$ points is an equivalence
class of planar figures (up to isotopy) in which $2n$ points
($n$ at the top and $n$ at the bottom of a rectangular shape)
are connected in pairs by non-crossing arcs. Here is an example
of a Temperley-Lieb diagram on $2\cdot 7$ points:
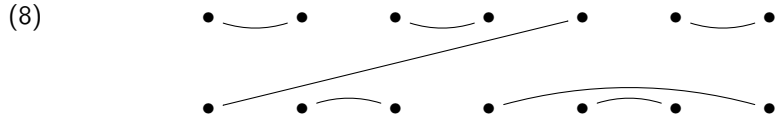
\begin{equation}\label{eqeq651}
\xymatrix{
\bullet\ar@/_4pt/@{-}[r]&\bullet&\bullet\ar@/_4pt/@{-}[r]&\bullet&
\bullet\ar@{-}[dllll]&\bullet\ar@/_4pt/@{-}[r]&\bullet\\
\bullet&\bullet\ar@/^4pt/@{-}[r]&\bullet&
\bullet\ar@/^8pt/@{-}[rrr]&\bullet\ar@/^4pt/@{-}[r]&\bullet&\bullet
}
\end{equation}
Temperley-Lieb diagrams on $2n$ points are enumerated by  the
$n$-th Catalan number $\mathbf{C}_n=\frac{1}{n+1}\binom{2n}{n}$.
Note that this coincides with the number of fully commutative
elements in $S_n$. 

The {\em Temperley-Lieb algebra} $\mathbf{TL}_n(v)$ is defined as
the $\mathbb{A}$-algebra whose basis consists of all 
Temperley-Lieb diagrams on $2n$ points and the product is given by
first concatenating the diagrams and then applying a certain 
straightening procedure which involves removal of closed connected
components at the expense of multiplication with $v+v^{-1}$,
see \cite{Mar} for details. In fact, $\mathbf{TL}_n(v)$ is the 
quotient of $\mathbf{H}_n$ by the two-sided ideal generated by
all $\underline{H}_w$, where $w$ is {\em not} fully commutative.
The basis of the Temperley-Lieb diagrams is exactly the image of
the Kazhdan-Lusztig basis of $\mathbf{H}_n$, namely, 
those elements in the Kazhdan-Lusztig basis that correspond
to fully commutative indices, under this projection.
Consequently, sending a fully commutative $w\in S_n$ to the
image of $\underline{H}_w$ in $\mathbf{TL}_n(v)$ under the above 
projection is a bijection between the fully commutative elements
of $S_n$ and Temperley-Lieb diagrams on $2n$ points. We adopt the
convention that in the pictorial representation of the 
product $ab$ of two Temperley-Lieb diagrams,
the diagram of $a$ should be placed below the diagram of $b$.

Given a Temperley-Lieb diagram, an arc connecting a point of the
top row with a point of the bottom row is called {\em propagating}.
An arc connecting a point of the
top row with a point of the top row is called a {\em cup},
and an arc connecting a point of the
bottom row with a point of the bottom row is called a {\em cap}.
In diagrammatic terms, KL cells have the following description,
see \cite{MMM}:
\begin{itemize}
\item two  Temperley-Lieb diagrams are left related if and only
if they have the same cups;
\item two  Temperley-Lieb diagrams are right related if and only
if they have the same caps;
\item two  Temperley-Lieb diagrams are two-sided related if and only
if they have the same number of propagating arcs.
\end{itemize}
Finally, the main result of \cite{MMM} has the following diagrammatic
reformulation: a fully commutative $w\in S_n$ is Kostant negative
if and only if there are two consecutive points in the top row
of the corresponding Temperley-Lieb diagram such that the left one of these
two points is connected by a cup to some point to the left of it 
and the right one of these two points is  connected by a cup to some point on 
the right of it. Such two cups are called {\em not nested} (i.e. they are 
not inside one another). For example, the fully commutative element of
$S_7$ represented by the diagram \eqref{eqeq651} is Kostant
negative since the two leftmost cups are not nested. 
An example of two nested caps can be found
on the same diagram \eqref{eqeq651} in the bottom right corner.

Now let us take another look at the elements $\tau_{k,a}$. 
Here are the Temperley-Lieb diagrams corresponding to 
$\tau_{4,1}$
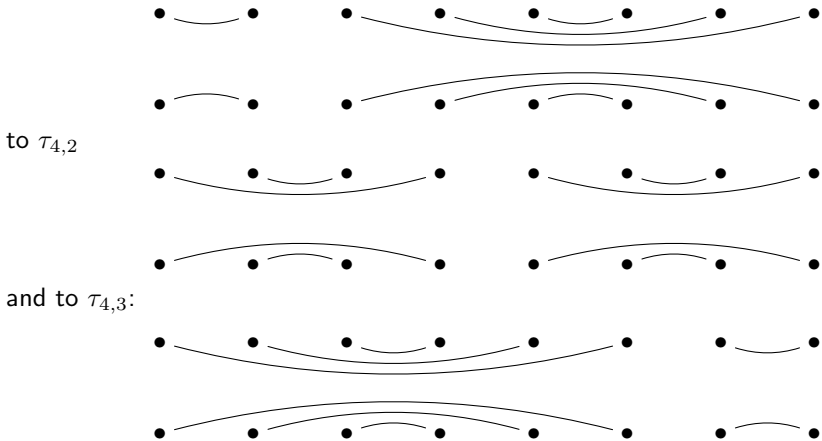
\begin{displaymath}
\xymatrix{
\bullet\ar@/_4pt/@{-}[r]&\bullet&
\bullet\ar@/_12pt/@{-}[rrrrr]&\bullet\ar@/_8pt/@{-}[rrr]
&\bullet\ar@/_4pt/@{-}[r]&\bullet&\bullet&\bullet\\
\bullet\ar@/^4pt/@{-}[r]&\bullet&
\bullet\ar@/^12pt/@{-}[rrrrr]&\bullet\ar@/^8pt/@{-}[rrr]&
\bullet\ar@/^4pt/@{-}[r]&\bullet&\bullet&\bullet
}
\end{displaymath}
to $\tau_{4,2}$
\begin{displaymath}
\xymatrix{
\bullet\ar@/_8pt/@{-}[rrr]&\bullet\ar@/_4pt/@{-}[r]&\bullet&\bullet
&\bullet\ar@/_8pt/@{-}[rrr]
&\bullet\ar@/_4pt/@{-}[r]&\bullet&\bullet\\
\bullet\ar@/^8pt/@{-}[rrr]&\bullet\ar@/^4pt/@{-}[r]&\bullet&
\bullet&\bullet\ar@/^8pt/@{-}[rrr]&
\bullet\ar@/^4pt/@{-}[r]&\bullet&\bullet
}
\end{displaymath}
and to $\tau_{4,3}$:
\begin{displaymath}
\xymatrix{
\bullet\ar@/_12pt/@{-}[rrrrr]&\bullet\ar@/_8pt/@{-}[rrr]
&\bullet\ar@/_4pt/@{-}[r]&\bullet&\bullet&\bullet&
\bullet\ar@/_4pt/@{-}[r]&\bullet\\
\bullet\ar@/^12pt/@{-}[rrrrr]&\bullet\ar@/^8pt/@{-}[rrr]&
\bullet\ar@/^4pt/@{-}[r]&\bullet&\bullet&\bullet
&\bullet\ar@/^4pt/@{-}[r]&\bullet
}
\end{displaymath}
Looking at these examples and the definitions, is easy to see that 
the elements $\tau_{k,a}$ are exactly those Kostant negative involutions 
which contain a unique pair of non-nested cups and, additionally, the
outer ends of these non-nested cups are the extreme points $1$ and $n$.

Now let $w\in S_n$ be a fully commutative Kostant cuspidal involution. Let 
$u\in S_{n-1}$ be the fully commutative element that corresponds to
the consecutive pattern of length $n-1$ in $w$ which starts at position $1$
(in other words, it is obtained from $w$ by removing the last element 
$w(n)$). We need to determine how the Temperley-Lieb diagram $b$ of $u$
can be obtained from the the Temperley-Lieb diagram $a$ of $w$. Note that
our Kostant cuspidality assumption implies that $w$ (equivalently, $a$) 
is Kostant negative while $u$ (equivalently, $b$) is Kostant positive.

The rightmost point in the top row of $a$ (the one which corresponds to $w(n)$)
is either the endpoint of a cup or of a propagating line. In the second case,
the propagating line must be simply vertical as $w$ is an involution
(and hence is symmetric with respect to the flip that swaps the top and the
bottom rows). This means that $w(n)=n$ and hence $u$ is the involution
obtained by restricting $w$ to $\{1,2,\dots,n-1\}$. This means that $b$
is obtained from $a$ by removing the vertical propagating line connecting
the rightmost points in both rows. In particular, if $a$ was Kostant negative,
so is $b$ (as the Kostant negativity condition is not affected by such removal).
Therefore this case cannot occur.

In the first case, the rightmost point in the top row of $a$ is connected
by a cup to some other point, say number $i$, counted from the left, in 
the top row. Then we can view $a$ as the juxtaposition of two diagrams:
the diagram $a'$ corresponding to the elements $1,2,\dots,i-1$ in both rows,
and the diagram $a''$ corresponding to the elements $i,i+1,\dots,n$.
The diagram $a'$, clearly, survives when we go from $a$ to $b$.
In particular, if $a'$ is Kostant negative, then so is $b$.
Therefore $a'$ must be Kostant positive.

Assume now that $i>1$. Then,
by symmetry (coming from swapping the order of our elements, that is, 
considering the consecutive pattern of length $n-1$ starting from $2$),
we similarly obtain that $a''$ is Kostant positive. But if both $a'$
and $a''$ are Kostant positive, the
only way for $a$ itself to be Kostant negative is that the cup
connecting $i$ and $n$ has a non-nested counterpart connecting $i-1$
with some $j<i-1$. From the above arguments, it follows that $j=1$
and thus $w=\tau_{n,\frac{i-1}{2}}$. 

We are left with the case when $i=1$. Note that, in this case, $a$ has no
propagating lines. Also note that, in this case, the cup connecting 
the first and the last points in the top row cannot be a part of a pair
of two cups that are not nested. Therefore the Kostant negativity of
$w$ implies that there exists a pair of other non-nested cups in $w$.
Let us assume that the endpoints of this pair of non-nested cups
are $s,t,t+1,r$ such that $1<s<t<t+1<r<n$. 
Then, from the explicit construction
in \cite[Subsection~5.5]{MMM} it follows that there are
different $x,y\in S_n$ such that $\theta_x L_w\cong \theta_y L_w\neq 0$
and, additionally, that $x(1)=1$, $x(n)=n$, $y(1)=1$ and $y(n)=n$
(as the construction described in \cite[Subsection~5.5]{MMM} 
changes the outer cup connecting $1$ and $n$ into a pair of
two vertical lines connecting the leftmost and the rightmost points, respectively).
Put differently, both $x$ and $y$ belong to each of the two parabolic subgroups 
of $S_n$ obtained by removing $s_1$ or $s_{n-1}$ from the set of all simple
reflections in $S_n$.
But then $\theta_x L_u\cong \theta_y L_u\neq 0$ by
\cite[Theorem~37]{CMZ} which implies that $u$ is Kostant negative, a contradiction.

This completes the proof 
of Theorem~\ref{thm-s7.3-1}.

\section{Three new infinite families of Kostant cuspidal involutions}\label{s5}

\subsection{First infinite family of examples}\label{s5.1}

\begin{theorem}\label{thm-cusp1}
For $n\geq 5$, the element
$w:=\mathrm{inv}(2,n-1)$ of $S_n$ is Kostant cuspidal. 
\end{theorem}

\subsection{Proof of Theorem~\ref{thm-cusp1}}\label{s5.2}

We split our proof of Theorem~\ref{thm-cusp1} in several steps.
As the first step, we will show that any proper consecutive
pattern of $w$  is Kostant positive.
Because of Theorem~\ref{thm-main}, it is enough to check the
patterns obtained from $w$ by removing either the first $1$
or the last $n$. Up to the outer automorphism of the root system
(given by the flip of the Dynkin diagram), we only need
to consider one element, namely, $\mathrm{inv}(1,n-2)\in S_{n-1}$.
We formulate it as the following lemma:

\begin{lemma}\label{lem-s4.3-1}
For every $m\geq 2$, the element 
$\mathrm{inv}(1,m-1)\in S_{m}$ is Kostant positive.
\end{lemma}

\begin{proof}
For the element $\mathrm{inv}(1,m-1)$, the corresponding tableaux 
$P_{\mathrm{inv}(1,m-1)}$ and $Q_{\mathrm{inv}(1,m-1)}$ 
are both equal to
\begin{displaymath}
\begin{ytableau}
\scriptstyle{1}&
\scriptstyle{m}\\
\scriptstyle{2}\\
\scriptstyle{\vdots}\\
\resizebox{4mm}{!}{$\scriptstyle{m-1}$}
\end{ytableau}
\end{displaymath}
Consider the element $u:=m(m-1)\dots4312\in S_m$.
It is easy to see that, for this element $u$, we have 
$Q_u=Q_{\mathrm{inv}(2,n-1)}$, so that $u\sim_L \mathrm{inv}(1,m-1)$.
Also note that we can write $u=\mathrm{inv}(1,2)\mathrm{inv}(1,m)$, that is,
$u$ has the form $w'_0w_0$, where $w_0=\mathrm{inv}(1,m)$ is the
longest element in $S_m$ and $w'_0=\mathrm{inv}(1,2)$ is the 
longest element of a parabolic subgroup of $S_m$. Therefore
$u$ is Kostant positive by \cite[Theorem~4.4]{GJ}.
Consequently, $\mathrm{inv}(1,m-1)$
is Kostant positive by \cite[Theorem~V]{MS}.
\end{proof}

Now we need to prove that the original element $w$ is Kostant negative.
If, for some $x\in S_n$, the module $\theta_x L_w$ is non-zero and 
decomposable, we automatically have that $w$ is Kostant negative
and we are done.
Therefore let us assume that all non-zero modules of the form
$\theta_x L_w$ are indecomposable. We will show that, in this case,
there exist different $x,y\in S_n$ such that 
$\theta_x L_w\cong \theta_y L_w\neq 0$. In fact, we will explicitly
construct such $x$ and  $y$.

Consider the following element $d\in S_n$, given using the one-line notation:
\begin{displaymath}
{\color{teal}(n-2)(n-1)(n-3)(n-4)\dots 4312}n,
\end{displaymath}
here the {\color{teal}non-trivial part} of the permutation is given in teal.
Note that the definition of $d$ requires $n\geq 5$.
Clearly, $d^2=e$.

\begin{lemma}\label{lem-s4.3-2}
We have $ds_{1}\geq_R d$ and $ds_{1}\sim_R w$, in particular,
$w\geq_R d$. 
\end{lemma}

\begin{proof}
The element $ds_{1}$ equals
\begin{displaymath}
{\color{purple}(n-1)(n-2)(n-3)(n-4)\dots 4312}n
\end{displaymath}
and hence has one inversion more than $d$. Therefore $ds_{1}>d$
implying $ds_{1}\geq_R d$. It remains to 
show that $ds_{1}\sim_R w$. This follows by observing that 
their insertion tableaux $P_w$ and $P_{ds_{1}}$ are both equal to
the following tableau:
\begin{equation}\label{eq-s3.3-3-n}
\begin{ytableau}
\scriptstyle{1}&
\scriptstyle{2}&
\scriptstyle{n}\\
\scriptstyle{3}\\
\scriptstyle{4}\\
\scriptstyle{\vdots}\\
\resizebox{4mm}{!}{$\scriptstyle{n-2}$}\\
\resizebox{4mm}{!}{$\scriptstyle{n-1}$}
\end{ytableau}
\end{equation}
\end{proof}

From $w\geq_R d$, using \cite[Formula~(1)]{KiM}, 
it follows that $\theta_d L_w\neq 0$ since both $w$
and $d$ are involutions. Our plan is to show that 
$\theta_d L_w\cong \theta_{s_{n-2}s_{n-1}d} L_w$.
As, clearly, we have $d\neq s_{n-2}s_{n-1}d$, this will
complete the proof of Theorem~\ref{thm-cusp1}.
To prove the isomorphism $\theta_d L_w\cong \theta_{s_{n-2}s_{n-1}d} L_w$,
we would need a few  auxiliary lemmata.

\begin{lemma}\label{lem-s4.3-3}
The module $\theta_{s_{n-2}s_{n-1}} L_w$ is 
isomorphic to $\theta_{s_{n-1}}L_{ws_{n-1}}$.
\end{lemma}

\begin{proof}
We have $\theta_{s_{n-2}s_{n-1}}=\theta_{s_{n-1}}\circ\theta_{s_{n-2}}$. Consider the
module $\theta_{s_{n-2}} L_w$. Note that $ws_{n-2}<w$, so this module is non-zero.
Both the top and the socle of this module is isomorphic to $L_w$ and 
the latter module is killed by $s_{n-1}$ as $ws_{n-1}>w$. Therefore 
$\theta_{s_{n-2}s_{n-1}} L_w$ is isomorphic to
$\theta_{s_{n-1}}\mathbf{J}_{s_{n-2}}L_w$, where
$\mathbf{J}_{s_{n-2}}L_w$ is the Jantzen middle of $\theta_{s_{n-2}} L_w$. 

The module $\mathbf{J}_{s_{n-2}}L_w$ has $L_{ws_{n-1}}$ as a summand. 
Indeed, we have $\mu(w,ws_{n-1})=1$. Further,
$ws_{n-1}s_{n-2}>ws_{n-1}$. Therefore 
$[\mathbf{J}_{s_{n-2}}L_w:L_{ws_{n-1}}]=1$.
As $ws_{n-1}>w$, it follows that $\theta_{s_{n-1}}L_{ws_{n-1}}\neq 0$
is a summand of $\theta_{s_{n-1}}\mathbf{J}_{s_{n-2}}L_w$.
So, to complete the proof of our lemma, we just need to argue that
$\theta_{s_{n-2}s_{n-1}} L_w$ is indecomposable.
By \cite[Proposition~2]{CMZ}, the module $\theta_{s_{n-2}s_{n-1}} L_w$ is indecomposable
if and only if the module $\theta_{s_{n-2}} L_w$ is indecomposable. The module
$\theta_{s_{n-2}} L_w$ has simple top and thus is indecomposable.
This completes the proof.
\end{proof}

\begin{lemma}\label{lem-s4.3-4}
We have $d\sim_L s_{n-2}s_{n-1}d$. 
\end{lemma}

\begin{proof}
The element $s_{n-2}s_{n-1}d$ is given by 
the following:
\begin{displaymath}
{\color{teal}(n-1)n(n-3)(n-4)\dots 4312
(n-2)}.
\end{displaymath}
It is easy to check that the recording tableaux
$Q_{d}$ and $Q_{s_{n-2}s_{n-1}d}$ are both equal to
\begin{equation}\label{eq-s3.3-35}
\begin{ytableau}
\scriptstyle{1}&
\scriptstyle{2}&
\scriptstyle{n}\\
\scriptstyle{3}&\resizebox{4mm}{!}{$\scriptstyle{n-1}$}\\
\scriptstyle{4}\\
\scriptstyle{\vdots}\\
\resizebox{4mm}{!}{$\scriptstyle{n-2}$}
\end{ytableau}
\end{equation}
This implies the claim.
\end{proof}

\begin{lemma}\label{lem-s4.3-5}
We have $s_{n-2}s_{n-1}d\leq_R w$. 
\end{lemma}

\begin{proof}
Let $\tilde{d}$ be the involution such that 
$\tilde{d}\sim_R s_{n-2}s_{n-1}d$. Then one checks
that $\tilde{d}$ is given by
\begin{displaymath}
1{\color{teal}(n-1)(n-3)(n-4)\dots 43n2(n-2)}.
\end{displaymath}
The element $\tilde{d}s_{1}$ is thus given by 
\begin{displaymath}
{\color{teal}(n-1)1(n-3)(n-4)\dots 43n2(n-2)}
\end{displaymath}
and therefore $\ell(\tilde{d}s_{1})>\ell(\tilde{d})$. This implies
$\tilde{d}s_{1}\geq_R \tilde{d}$. It is easy to check that the
insertion tableau of $\tilde{d}s_{1}$ 
coincides with the insertion tableau of $w$
(this is given by \eqref{eq-s3.3-3-n}).
Therefore $\tilde{d}s_{1}\sim_R w$. This implies the claim.
\end{proof}

Combining Lemma~\ref{lem-s4.3-5} with 
\cite[Formula~(1)]{KiM}, we conclude that we have the inequality
$\theta_{s_{n-2}s_{n-1}d} L_w\neq 0$.  

\begin{lemma}\label{lem-s3.3-9}
We have $d\not \leq_R (ws_{n-1})^{-1}=s_{n-1}w$. 
\end{lemma}

\begin{proof}
The tableau $Q_{(ws_{n-1})^{-1}}=P_{s_{n-1}w}$ looks as follows:
\begin{equation}\label{eq-s3.3-333}
\begin{ytableau}
\scriptstyle{1}&
\scriptstyle{2}&
\resizebox{4mm}{!}{$\scriptstyle{n-1}$}\\
\scriptstyle{3}\\
\scriptstyle{4}\\
\scriptstyle{\vdots}\\
\resizebox{4mm}{!}{$\scriptstyle{n-2}$}\\
\scriptstyle{n}
\end{ytableau}
\end{equation}
We want to compare the standard Young tableaux 
\eqref{eq-s3.3-35} and \eqref{eq-s3.3-333}. 
Note that the shape of \eqref{eq-s3.3-35} dominates 
the shape of \eqref{eq-s3.3-333}. Now,
in both, consider the following respective 
standard Young subtableaux with content $\{1,2,\dots,n-1\}$:
\begin{displaymath}
\begin{ytableau}
\scriptstyle{1}&
\scriptstyle{2}\\
\scriptstyle{3}&\resizebox{4mm}{!}{$\scriptstyle{n-1}$}\\
\scriptstyle{4}\\
\scriptstyle{\vdots}\\
\resizebox{4mm}{!}{$\scriptstyle{n-2}$}
\end{ytableau}
\qquad\qquad
\begin{ytableau}
\scriptstyle{1}&
\scriptstyle{2}&
\resizebox{4mm}{!}{$\scriptstyle{n-1}$}\\
\scriptstyle{3}\\
\scriptstyle{4}\\
\scriptstyle{\vdots}\\
\resizebox{4mm}{!}{$\scriptstyle{n-2}$}
\end{ytableau}
\end{displaymath}
Note that the shape of the second one dominates the 
shape of the first.  Combining the two observations,
we conclude that the standard Young tableaux 
\eqref{eq-s3.3-35} and \eqref{eq-s3.3-333}
are not comparable with respect to the dominance 
order on standard Young tableaux. From 
\cite[Subsection~3.3]{Ta} it now follows that the
corresponding elements are not comparable with respect
to the right KL order. This completes the proof.
\end{proof}

Now Lemmata~\ref{lem-s4.3-4} and \ref{lem-s4.3-5} motivate the following:

\begin{corollary}\label{cor-s4.3-6}
We have $\theta_d L_w\cong \theta_{s_{n-2}s_{n-1}d} L_w$.
\end{corollary}

\begin{proof}
We now know that both $\theta_{d} L_w\neq 0$
and $\theta_{s_{n-2}s_{n-1}d} L_w\neq 0$.
Further, we know that $\theta_{s_{n-2}s_{n-1}d}$ is a summand of
$\theta_{d}\circ\theta_{s_{n-2}s_{n-1}}$ and that 
$\theta_{s_{n-2}s_{n-1}}L_w\cong \theta_{s_{n-1}}L_{ws_{{n-1}}}$.

By Lemma~\ref{lem-s3.3-9} and  \cite[Formula~(1)]{KiM}, we have
$\theta_d L_{ws_{n-1}}=0$. Also, $L_w$ is a summand of 
the Jantzen middle $\mathbf{J}_{s_{n-1}}(L_{ws_{n-1}})$
appearing with multiplicity one since $w<ws_{n-1}$.
By \cite[Proposition~2]{CMZ}, the module
$\theta_{s_{n-2}s_{n-1}d} L_w$  has the same number of summands as 
$\theta_d L_w$.
Therefore, to complete the proof,
we just need to show that $\theta_d$ kills all other potential
summands $L_y$ of $\mathbf{J}_{s_{n-1}}(L_{ws_{n-1}})$.
So, we assume that $L_y$ is a summand of $\mathbf{J}_{s_{n-1}}(L_{ws_{n-1}})$
different from $L_w$ and such that $\theta_d L_y\neq 0$.

Each such $y$ satisfies both $ys_{n-1}>y$ and $y\leq_R w$
(for the latter, see \cite[Lemma~13]{MM1}). 
Moreover, any top of the non-zero module $\theta_d L_y$
is a top of a module of the form $\theta L_w$, where $\theta$ is
a projective functor. Therefore $y\sim_R w$, see \cite[Corollary~14]{MM1}.

We have $s_k d<d$, for $k=2,3,\dots,n-3$. This means that 
$\theta_d\circ \theta_{s_k}\cong\theta_d\oplus\theta_d$, for all such $k$.
If, for some such $k$, we would have $ys_k>y$, then 
$\theta_{s_k}L_y=0$ and hence 
\begin{displaymath}
\theta_d L_y\oplus \theta_d L_y\cong
\theta_d\circ\theta_{s_k}(L_y)\cong 0.
\end{displaymath}
Therefore $ys_k<y$, for all such $k$.

The only element $y\sim_R w$ which satisfies $y\neq w$ and
both $ys_{n-1}>y$ and $ys_k<y$, for all $k=2,3,\dots,n-3$, is the following
element:
\begin{displaymath}
{\color{teal}(n-1)(n-2)(n-3)\dots 4312}n.
\end{displaymath}
Note that this element is longer than $ws_{n-1}$ but does not contain
$s_{n-1}$ as a factor in any reduced expression. Therefore this elements is not
related to $ws_{n-1}$ in the Bruhat order and hence $\mu(y,ws_{n-1})=0$.
This means that such $L_y$ do not appear in $\mathbf{J}_{s_{n-1}}(L_{ws_{n-1}})$.
This completes the proof of our corollary.
\end{proof}

From Corollary~\ref{cor-s4.3-6}, we obtain
$\mathbf{Kh}^{(III)}(w)=-$ and hence 
$\mathbf{Kh}^{(II)}(w)=-$ as well.
Consequently, $\mathbf{K}(w)=-$
by \cite[Proposition~8.14]{KMM}.
This completes the proof of Theorem~\ref{thm-cusp1}.

\subsection{Second infinite family of examples}\label{s5.3}

\begin{theorem}\label{thm-cusp2}
For $n\geq 7$, the element
$32145\dots (n-3)n(n-1)(n-2)$ (one-line notation) of $S_n$ is Kostant cuspidal.
\end{theorem}

The element mentioned in Theorem~\ref{thm-cusp2} above has a reduced
expression of length six, that is: $s_1s_2s_1s_{n-2}s_{n-1}s_{n-2}$.

\subsection{Proof of Theorem~\ref{thm-cusp2}}\label{s5.4}

In order to prove Theorem~\ref{thm-cusp2}, we only need to prove two things:
\begin{itemize} 
\item that the element $u_n:=s_1s_2s_1s_{n-2}s_{n-1}s_{n-2}$
is Kostant negative, for $n\geq 7$;
\item as well as that the element 
$v_n:=s_1s_2s_1s_{n-1}$ is Kostant positive, for $n\geq 6$.
\end{itemize}
This is due to the fact that the two consecutive patterns 
of length $n-1$ in $u_n$ are both equal to $v_{n-1}$, up to the
symmetry of the root system. We start with the second claim.

\begin{proposition}\label{prop-s5.4-n1}
For $n\geq 6$, the element  $v_n:=s_1s_2s_1s_{n-1}$ is Kostant positive.
\end{proposition}

\begin{proof}
Let $w_0$ be the longest element in $S_n$
and $w'_0$ the longest element in the parabolic subgroup
$S_{n-1}$ of $S_n$ consisting of all permutations that fix $n$.
Consider the element 
$w:=s_1s_{n-2}\in S_{n-1}$. We claim that $v_n\sim_L ww'_0w_0$.
It is easy to see that the recording tableau of $v_n$ is
\begin{equation}\label{eqeq-s5v}
\begin{ytableau}
\scriptstyle{1}&\scriptstyle{4}&\scriptstyle{5}&\scriptstyle{\cdots}
&\resizebox{4mm}{!}{$\scriptstyle{n-1}$}\\
\scriptstyle{2}&\scriptstyle{n}\\
\scriptstyle{3}
\end{ytableau}
\end{equation}
In the one-line notation, the element $ww'_0w_0$ is
\begin{displaymath}
n213...(n-4)(n-3)(n-1)(n-2) 
\end{displaymath}
and it is straightforward to check that the 
recording tableau of this element is given by \eqref{eqeq-s5v} as well.
This prove that $v_n\sim_L ww'_0w_0$.

Since $n\geq 6$, the element $w$ is Kostant positive by 
\cite[Theorem~5.1]{MMM} (using the terminology of \cite[Theorem~5.1]{MMM},
the element $w$ is the product of two distant elementary transpositions). 
Hence $ww'_0w_0$ is Kostant positive by
\cite[Theorem~1.1]{Ka}. Consequently, $v_n$ is Kostant positive 
by the left cell invariance.
\end{proof}

To prove that $u_n$ is Kostant negative will require significantly more
effort and a lot of preparation.

Define $y\in S_n$ as the transposition $(2,n-1)$.
Our aim is to prove that there is an isomorphism 
$\theta_{s_1y}L_{u_n}\cong \theta_{s_{n-1}y}L_{u_n}\neq 0$. 
This would imply that  $u_n$ is Kostant negative by \cite[Proposition~8.14]{KMM}.
Define $x\in S_n$ as the following permutation
(in one-line notation):
\begin{displaymath}
{\color{teal}3n2}45\dots (n-4)(n-3){\color{teal}(n-1)1(n-2)}. 
\end{displaymath}
The {\color{teal}teal} parts specifies the elements that are
displaced by $x$.

We will now establish a few technical lemmata about these elements.
All these statements were suggested following small rank computations
with GAP3. Recall that $s_i=(i,i+1)$, where $i=1,2,\dots,n-1$,
are the simple reflections in $S_n$.

\begin{lemma}\label{lem-s7.4-25}
If $n\geq 3$, then, for any $a\leq (2,n)$,
we have $\mu(a,(1,n))=0$.
\end{lemma}

\begin{proof}
Since $\ell((2,n))<\ell((1,n))$ and $a\leq (2,n)$,
we have $\ell(a)<\ell((1,n))$. Note that $s_1$
belongs to the left descent set of $(1,n)$ but
not to the left descent set of $a$. Therefore
$\mu(a,(1,n))=0$ by \cite[Fact~3.2]{Wa}.
\end{proof}

\begin{lemma}\label{lem-s7.4-2}
For $i,j\in\{1,2,\dots,n-1\}$ such that $1<i<j-1$, we have: 
\begin{displaymath}
\theta_{(i,j)}\theta_{s_{i-1}}\cong\theta_{s_{i-1}(i,j)},\quad
\theta_{s_{i-1}}\theta_{(i,j)}\cong\theta_{(i,j)s_{i-1}} \quad
\end{displaymath}
and
\begin{displaymath}
\theta_{(i,j)s_{i-1}}\theta_{s_{i-1}}\cong
\theta_{s_{i-1}}\theta_{s_{i-1}(i,j)}\cong
\theta_{(i-1,j)} .
\end{displaymath}
\end{lemma}

\begin{proof}
From \eqref{eq:kl-mu2}, for $w\in W$ and a simple reflection $s$, we have 
\begin{equation}\label{eq-s7.4.123}
\theta_s\circ \theta_w\cong
\begin{cases}
\theta_w\oplus \theta_w,& ws<w;\\
\displaystyle
\theta_{ws}\oplus \bigoplus_{u<w,us<u}\theta_u^{\oplus \mu(u,w)},& ws>w;
\end{cases}
\end{equation}
and similarly for the multiplication with  $\theta_s$
on the other side (using \eqref{eq:kl-mu}).

Now note that $s_{i-1}$ does not appear in any reduced expression of
$(i,j)$. If we use \eqref{eq-s7.4.123} to compute
$\theta_{s_{i-1}}\theta_{(i,j)}$, we get that no $u<(i,j)$ satisfies 
$us_{i-1}<u$ so that the corresponding sum becomes empty. This implies
that $\theta_{s_{i-1}}\theta_{(i,j)}$ is isomorphic to 
$\theta_{(i,j)s_{i-1}}$. Similarly, we have that 
$\theta_{(i,j)}\theta_{s_{i-1}}$ is isomorphic to 
$\theta_{s_{i-1}(i,j)}$.

Let us now prove that $\theta_{s_{i-1}}\theta_{s_{i-1}(i,j)}\cong
\theta_{(i-1,j)}$, the proof of the remaining formula
$\theta_{(i,j)s_{i-1}}\theta_{s_{i-1}}\cong
\theta_{(i-1,j)}$ is similar. Using  \eqref{eq-s7.4.123}, we
need to understand all $a<s_{i-1}(i,j)$ such that 
$as_{i-1}<a$. Then $a$ cannot have any $s_{i}$ in its
reduced expression. This implies that $a=s_{i-1}a'$, where
$a'<(i+1,j)$. 

It is easy to see that multiplication with $s_{i-1}$
defines a bijection between the Bruhat intervals 
$[u,s_{i-1}(i,j)]$ and $(u',(i,j))$. Moreover, this
bijection preserves the preimage of $j$ for all elements.
From \cite[Theorem~3.7]{BBDVW} 
it thus follows that the Kazhdan-Lusztig polynomials
$p_{a,s_{i-1}(i,j)}$ and $p_{a',(i,j)}$ coincide.
From Lemma~\ref{lem-s7.4-25}, we have
$\mu(a',(i,j))=0$, and the claim follows.
\end{proof}

\begin{lemma}\label{lem-s7.4-3}
We have $\theta_{y}L_{u_n}=0$. 
\end{lemma}

\begin{proof}
Using Lemma~\ref{lem-s7.4-2}, we can write $\theta_{y}$ as
\begin{displaymath}
\theta_{s_{2}}\theta_{s_{3}}\dots
\theta_{s_{n-4}}\theta_{(n-3,n-1)}
\theta_{s_{n-4}}\dots\theta_{s_{3}}\theta_{s_{2}}
\end{displaymath}
and identify it as a summand of 
\begin{displaymath}
\theta_{s_{2}}\theta_{s_{3}}\dots
\theta_{s_{n-4}}\theta_{s_{n-3}}\theta_{s_{n-2}}\theta_{s_{n-3}}
\theta_{s_{n-4}}\dots\theta_{s_{3}}\theta_{s_{2}}.
\end{displaymath}

The module $\theta_{s_{2}}L_{u_n}$ has $L_{u_n}$
as both simple top and socle. It also has $L_{u_ns_3}$ as
a subquotient as $\mu(u_n,u_ns_3)=1$ since the 
elements $u_n$ and $u_ns_3$ are Bruhat neighbours. 

Note that $\theta_{s_{3}}\theta_{s_{2}}\cong \theta_{s_{2}s_{3}}$
by \cite[Theorem~1]{BW} and that $\theta_{s_{2}s_{3}}L_{u_n}$
is indecomposable by \cite[Subsection~3.3]{CMZ}. Since
$\ell(u_ns_3)>\ell(u_n)$, we have 
$\theta_{s_{3}}L_{u_n}=0$ and hence $L_{u_ns_3}$
must be the only subquotient of $\theta_{s_{2}}L_{u_n}$
that is not killed by $\theta_{s_{3}}$. This implies that 
$\theta_{s_{2}s_{3}}L_{u_n}$ is isomorphic to
$\theta_{s_{3}}L_{u_ns_3}$.

Proceeding inductively, we similarly get that
\begin{displaymath}
\theta_{s_{2}s_{3}\dots s_{n-2}} L_{u_n}\cong
\theta_{s_{n-2}}L_{u_ns_2s_{3}\dots s_{n-2}},
\end{displaymath}
in particular, this latter module has Loewy length $3$. 
Therefore, applying now $\theta_{s_{n-3}}$, we get a module of Loewy length
at most $5$. 

By Lemma~\ref{lem-s7.4-2}, this new module should
be isomorphic to
\begin{displaymath}
\theta_{s_{2}s_{3}\dots s_{n-3}}L_{u_n}\oplus
\theta_{s_{2}s_{3}\dots s_{n-3}s_{n-2}s_{n-3}}L_{u_n}. 
\end{displaymath}
Now we note that the value of Lusztig's $\mathbf{a}$-function
at the element 
\begin{displaymath}
s_{2}s_{3}\dots s_{n-3}s_{n-2}s_{n-3} 
\end{displaymath}
is $3$ (as it has the factor $s_{n-3}s_{n-2}s_{n-3}$).
In particular, from \cite[Corollary~7]{Ma2} it follows that 
the Loewy length of any non-zero module of the form
\begin{displaymath}
\theta_{s_{2}s_{3}\dots s_{n-3}s_{n-2}s_{n-3}}L_w
\end{displaymath}
must be at least $2\cdot3 +1=7$.  This means that
$\theta_{s_{2}s_{3}\dots s_{n-3}s_{n-2}s_{n-3}}L_{u_n}=0$ and therefore 
$\theta_{y}L_{u_n}=0$ follows from Lemma~\ref{lem-s7.4-3}, completing the proof.
\end{proof}

\begin{lemma}\label{lem-s7.4-4}
We have $\mu(x,{u_n})=1$, in particular,
$L_x$ appears as a subquotient of $\theta_{s_{1}}L_{u_n}$
with multiplicity $1$.
\end{lemma}

\begin{proof}
We can write $x=(u_ns_{2}s_{n-2})y$, where
$\ell(x)=\ell(u_ns_{2}s_{n-2})+\ell(y)$. Note that
$\ell(u_ns_{2}s_{n-2})=\ell(u_n)-2$.
Consider the elements
$x_1=s_1s_2s_1$ and $x_2=s_{n-2}s_{n-1}s_{n-2}$.
Then we have $x=(x_1s_{2})(x_2s_{n-2})y$ and in this decomposition
we have
$\ell(x)=\ell(x_1s_{2})+\ell(x_2s_{n-2})+\ell(y)$.

By \cite[Theorem~2]{SSV}, for $n\geq 3$, the Kazhdan-Lusztig polynomial
$p_{s_1s_{n-1},(1,n)}$ equals $(1+q)^{n-3}$. Consequently, the
Kazhdan-Lusztig polynomial $p_{s_{2}s_{n-2},y}$ equals $(1+q)^{n-5}$.

It is easy to see that multiplication with $x_1s_{2}$ provides
an isomorphism between the Bruhat intervals $[s_{2}s_{n-2},y]$
and $[x_1s_{n-2},x_1s_{2}y]$. Moreover, this isomorphism fixes
the preimage of $n$ for all elements. From \cite[Theorem~3.7]{BBDVW} 
it thus follows that $p_{x_1s_{n-2},x_1s_{2}y}=p_{s_{2}s_{n-2},y}$.

Similarly, it is easy to see that multiplication with $x_2s_{n-2}$ provides
an isomorphism between the Bruhat interval $[x_1s_{n-2},x_1s_{2}y]$
and $[u_n,x]$. Moreover, this isomorphism fixes
the preimage of $1$ for all elements. From \cite[Theorem~3.7]{BBDVW} 
it thus follows that $p_{x_1s_{n-2},x_1s_{2}y}=p_{u_n,x}$.

We have $\ell(u_n)=6$, $\ell(y)=2n-7$ and hence $\ell(x)=2n-4$.
It remains to observe that the degree $n-5$ of this polynomial
coincides with $\frac{\ell(x)-\ell(u_n)-1}{2}$ and 
that the leading term equals $1$. This implies
$\mu(x,{u_n})=1$.

Since $\ell(xs_{1})>\ell(x)$ and $\mu(x,{u_n})=1$, the module
$L_x$ appears as a subquotient of $\theta_{s_{1}}L_{u_n}$
with multiplicity $1$. This completes the proof.
\end{proof}

\begin{lemma}\label{lem-s7.4-6}
We have $\theta_{y}L_x\neq 0$.
\end{lemma}

\begin{proof}
As already mentioned in the proof of Lemma~\ref{lem-s7.4-4},
we can decompose the element $x$ as $(u_ns_{2}s_{n-2})y$, where
$\ell(x)=\ell(u_ns_{2}s_{n-2})+\ell(y)$.

Note that, for any $w\in W$, we have $\theta_w L_{w^{-1}}\neq 0$
since, by adjunction,
\begin{displaymath}
\mathrm{Hom}(P_e, \theta_w L_{w^{-1}}) \cong
\mathrm{Hom}(\theta_{w^{-1}}P_e,  L_{w^{-1}})\cong
\mathrm{Hom}(P_{w^{-1}},  L_{w^{-1}})\neq 0.
\end{displaymath}
Also, for $a,b\in W$ such that $\ell(a)+\ell(b)=\ell(ab)$,
we know that $\theta_{ab}$ is a summand of $\theta_b\theta_a$.

Therefore $\theta_{y}L_x= 0$ would imply
$\theta_{s_{2}s_{n-2}u_n}\theta_y L_x=0$ which, in turn, 
would imply $\theta_{x^{-1}}L_x=0$, a contradiction.
\end{proof}

\begin{lemma}\label{lem-s7.4-7}
We have $\theta_{s_{1}y} L_{u_n}\cong 
\theta_{s_{n-1}y} L_{u_n}\cong \theta_{y}L_x\neq 0$.
\end{lemma}

\begin{proof}
By symmetry, we just need to prove that 
$\theta_{s_{1}y} L_{u_n}\cong\theta_{y}L_x$.
From the above, we already know that 
$\theta_{s_{1}y}\cong \theta_{y}\theta_{s_{1}}$,
that $L_x$ appears with multiplicity $1$
in $\theta_{s_{1}}L_{u_n}$, that 
$\theta_{y}$ kills $L_{u_n}$ and that $\theta_y$
does not kill $L_x$. Therefore we just need
to argue that $\theta_{y}$ kills all simple subquotients of
$\theta_{s_{1}}L_{u_n}$ different from $L_x$.

Such a subquotient $L_z$ can either be from the same 
KL right cell as $u_n$ or from a strictly smaller
KL right cell. If $z$ is from a strictly smaller right cell,
then $\theta_{y}L_z=0$ for otherwise the module 
$\theta_{y}\theta_{s_{1}}L_{u_n}$ would have 
$\theta_{y}L_z$ as a non-trivial summand
(since $\theta_{y}L_{u_n}=0$). Any simple
top of $\theta_{y}L_z$  would be from the same 
KL right cell as $z$, which is not possible since any
simple top of $\theta_{y}\theta_{s_{1}}L_{u_n}$
must be from the same KL right cell as $u_n$.

It remains to consider the case when $z$ is from the same
KL right cell as $u_n$. In this case, the insertion 
tableau of $z$ and $u_n$ should be the same. Note that
the insertion tableau of $u_n$ is
\begin{equation}\label{eq-s3.3-3}
\begin{ytableau}
\scriptstyle{1}&
\scriptstyle{4}&
\scriptstyle{2}&
\scriptstyle{\dots}&
\resizebox{4mm}{!}{$\scriptstyle{n-2}$}\\
\scriptstyle{2}&
\resizebox{4mm}{!}{$\scriptstyle{n-1}$}\\
\scriptstyle{3}&
\scriptstyle{n}
\end{ytableau}
\end{equation}
This 
forces the following properties onto $z$:
\begin{itemize}
\item in the one-line notation, the elements
$3$, $2$ and $1$ in $z$ should be decreasing;
\item in the one-line notation, the elements
$n$, $n-1$ and $n-2$ in $z$ should be decreasing;
\item in any prefix of the  one-line notation,
the number of elements from the set $\{1,2,3\}$
should be greater than or equal to the number of 
elements from the set $\{n-2,n-1,n\}$.
\end{itemize}
One immediate consequence here is that $z(1)=3$ and 
$z(n)=n-2$.

Next we note that $\theta_{s_i}$ commutes with 
$\theta_{s_{1}}$, for $i=3,4,\dots,n-4$,
and also all such $\theta_{s_i}$ kill $L_{u_n}$
(as all such $s_i$ are not in the right descent of $u_n$).
Therefore all such $\theta_{s_i}$ should kill 
all simple subquotients of $\theta_{s_{1}}L_{u_n}$.
This means that, in the one-line notation for $z$, the 
elements in positions $3$, $4$,\dots, $n-2$ should be 
increasing. Together with the previous conditions,
this implies that $z(i)=i$, for all
$3<i<n-2$, moreover, that $z(3)\in\{1,2\}$
and $z(n-2)\in \{n-1,n\}$.

If $z(3)=1$, the above conditions force $z(2)=2$,
$z(n-2)=n$ and $z(n-1)=n-1$ and hence $z=u_n$,
which is not possible. Therefore $z(3)=2$ which forces
$z(2)=n$, $z(n-1)=1$ and $z(n-2)=n-1$. Hence $z=x$.
This completes the proof.
\end{proof}

Kostant negativity of $u_n$ follows directly from
Lemma~\ref{lem-s7.4-7} and \cite[Proposition~8.14]{KMM}.
This completes the proof of Theorem~\ref{thm-cusp2}.

\subsection{Third infinite family of examples}\label{s5.5}

Let $n\geq 6$ and $i\in\{3,4,\dots,n-3\}$. Denote by $\sigma_{n,i}$
the product of the two commuting transpositions $(1,i+1)$ and $(i,n)$.
Clearly, $\sigma_{n,i}$ is an involution. In one-line notation,
the element $\sigma_{n,i}$ is given by:
\begin{displaymath}
{\color{teal}(i+1)}23\dots (i-1){\color{teal}n1}(i+2)(i+3)\dots(n-1){\color{teal}i}, 
\end{displaymath}
where the displaced elements of $\{1,2,\dots,n\}$ are colored
in {\color{teal}teal}. The corresponding standard Young tableau is
\begin{displaymath}
 \begin{ytableau}
\scriptstyle{1}&
\scriptstyle{3}&
\dots&
\scriptstyle{i}&
\resizebox{4mm}{!}{$\scriptstyle{i+3}$}&
\dots&
\resizebox{4mm}{!}{$\scriptstyle{n-1}$}\\
\scriptstyle{2}&
\resizebox{4mm}{!}{$\scriptstyle{i+2}$}\\
\resizebox{4mm}{!}{$\scriptstyle{i+1}$}&
\scriptstyle{n}
\end{ytableau}
\end{displaymath}
provided that $i<n-3$ and 
\begin{displaymath}
 \begin{ytableau}
\scriptstyle{1}&
\scriptstyle{3}&
\dots&
\resizebox{4mm}{!}{$\scriptstyle{n-3}$}\\
\scriptstyle{2}&
\resizebox{4mm}{!}{$\scriptstyle{n-1}$}\\
\resizebox{4mm}{!}{$\scriptstyle{n-2}$}&
\scriptstyle{n}
\end{ytableau}
\end{displaymath}
if $i=n-3$.

\begin{theorem}\label{thm-cusp3}
For $n\geq 6$ and $i\in\{3,4,\dots,n-3\}$, the element
$\sigma_{n,i}\in S_n$ is Kostant cuspidal.
\end{theorem}

\subsection{Proof of Theorem~\ref{thm-cusp3}}\label{s5.6}

We start by proving that the two consecutive patterns in 
$\sigma_{n,i}$ of length $n-1$ are Kostant positive. We start
with the pattern in which $\sigma_{n,i}(n)$ is removed.
Such removal gives the one-line sequence
\begin{displaymath}
(i+1)23\dots (i-1)n1(i+2)(i+3)\dots(n-1) 
\end{displaymath}
and it corresponds to the $S_{n-1}$-pattern
\begin{displaymath}
\alpha:=i23\dots (i-1)(n-1)1(i+1)(i+2)\dots(n-2) 
\end{displaymath}
(here, for $s>i$, each $s$ of the original sequence is replaced by $s-1$).
The recording tableau for the latter element equals
\begin{equation}\label{eq-21-nn}
\begin{ytableau}
\scriptstyle{1}&
\scriptstyle{3}&
\dots&
\scriptstyle{i}&
\resizebox{4mm}{!}{$\scriptstyle{i+3}$}&
\dots&
\resizebox{4mm}{!}{$\scriptstyle{n-1}$}\\
\scriptstyle{2}&
\resizebox{4mm}{!}{$\scriptstyle{i+2}$}\\
\resizebox{4mm}{!}{$\scriptstyle{i+1}$}
\end{ytableau}
\end{equation}
Now consider the following permutation (one-line notation):
\begin{displaymath}
\beta:=(n-1)1356\dots (i+1)24(i+2)(i+3)\dots(n-3)(n-2) 
\end{displaymath}
It is easy to check that the recording tableau of this latter
element is also given by \eqref{eq-21-nn}. Therefore
$\alpha$ and $\beta$ belong to the same KL left cell
and hence $\mathbf{K}(\alpha)=\mathbf{K}(\beta)$.

The permutation $\beta\in S_{n-1}$ starts with $n-1$ and hence
its has the form $\gamma w_0^{(n-2)}w_0^{(n-1)}$, where  
$w_0^{(n-2)}$ and $w_0^{(n-1)}$ are the longest elements 
in $S_{n-2}$ and $S_{n-1}$, respectively,
and $\gamma\in S_{n-2}$ is the element
\begin{displaymath}
\gamma:= 1356\dots (i+1)24(i+2)(i+3)\dots(n-3)(n-2)
\end{displaymath}
By \cite[Theorem~1.1]{Ka}, we have $\mathbf{K}(\beta)=\mathbf{K}(\gamma)$.

The recording tableau of $\gamma$ equals:
\begin{displaymath}
\begin{ytableau}
\scriptstyle{1}&
\scriptstyle{2}&
\dots&
\resizebox{4mm}{!}{$\scriptstyle{i-1}$}&
\resizebox{4mm}{!}{$\scriptstyle{i+2}$}&
\resizebox{4mm}{!}{$\scriptstyle{i+3}$}&
\dots&
\resizebox{4mm}{!}{$\scriptstyle{n-2}$}
\\
\scriptstyle{i}&
\resizebox{4mm}{!}{$\scriptstyle{i+1}$}
\end{ytableau}
\end{displaymath}
The involution in $S_{n-4}$ with the same recording tableau
is $\delta=(i-2,i)(i-1,i+1)$ (given as the product of
two transpositions).
We have $\mathbf{K}(\delta)=\mathbf{K}(\gamma)$ since
$\delta$ and $\gamma$ belong to the same KL left cell.
Note that the insertion tableau of $\delta$ has two rows.
Hence $\delta$ is fully commutative. The Temperley-Lieb diagram
of $\delta$ is of the following form: 
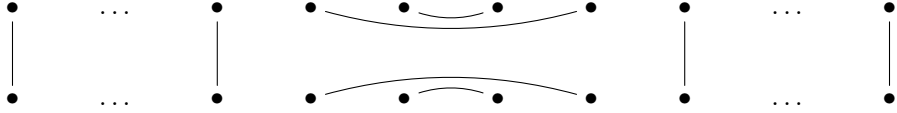
\begin{displaymath}
\xymatrix{
\bullet\ar@{-}[d]& \dots & \bullet\ar@{-}[d]
&\bullet\ar@/_8pt/@{-}[rrr]&\bullet\ar@/_4pt/@{-}[r]
&\bullet&\bullet&\bullet\ar@{-}[d]&
\dots & \bullet\ar@{-}[d]\\
\bullet& \dots & \bullet&\bullet\ar@/^8pt/@{-}[rrr]
&\bullet\ar@/^4pt/@{-}[r]&\bullet&\bullet&\bullet&
\dots & \bullet
}
\end{displaymath}
In particular, we see that the two cups in this diagram are nested. Therefore,
by \cite[Theorem~5.1]{MMM}, the element $\delta$ is Kostant positive.
From the above, we have $\mathbf{K}(\alpha)=\mathbf{K}(\delta)$, which means 
that $\alpha$ is Kostant positive.

We note that $\alpha$ has shape $(n-3,2,1)$. The general 
answer to Kostant's problem for elements of such shape
is in the process of being worked out in \cite{CM4}.

Next we go to the pattern in which $\sigma_{n,i}(1)$ is removed.
Note that the collection of all elements of the form
$\sigma_{n,i}$, for a fixed $n$ but all valid $i$, is invariant
under the symmetry of the root system (it swaps
$\sigma_{n,3}$ with $\sigma_{n,n-3}$, then 
$\sigma_{n,4}$ with $\sigma_{n,n-4}$ and so on).
Under this symmetry, the pattern which we get by removing 
$\sigma_{n,i}(1)$ from $\sigma_{n,3}$ corresponds to the 
pattern which we get by removing 
$\sigma_{n,i}(n)$ from $\sigma_{n,n-3}$ and so on.
Therefore, the removal of $\sigma_{n,i}(1)$, for all $i$,
corresponds to the removal of $\sigma_{n,i}(n)$, for all $i$.
The latter case is already treated above. Consequently, the
$n-1$-pattern obtained from $\sigma_{n,i}$ by removing $\sigma_{n,i}(1)$ 
is Kostant positive.

It remains to show that $\sigma_{n,i}$ itself is Kostant negative.
Define $x_{n,i},y_{n,i}\in S_n$ by their reduced expressions as follows:
\begin{displaymath}
x_{n,i}:=s_is_1s_2s_3\dots s_{i+1},\quad
y_{n,i}:=s_is_{n-1}s_{n-2}\dots s_{i-1}.
\end{displaymath}
In one-line notation, we have the following $x_{n,i}$
(with displaced elements of $\{1,2,\dots,n\}$ colored in {\color{teal}teal}):
\begin{displaymath}
{\color{teal}23\dots(i-1)(i+1)}i{\color{teal}(i+2)1}(i+3)(i+4)\dots n
\end{displaymath}
and the following $y_{n,i}$ (with displaced elements of $\{1,2,\dots,n\}$ 
colored in {\color{teal}teal}):
\begin{displaymath}
12\dots(i-2){\color{teal}n(i-1)}(i+1){\color{teal}i(i+2)(i+3)\dots(n-1)}.
\end{displaymath}
Note that both elements have the following insertion tableau and
hence belong to the same KL left cell:
\begin{displaymath}
\begin{ytableau}
\scriptstyle{1}&
\scriptstyle{2}&
\dots&
\resizebox{4mm}{!}{$\scriptstyle{i-1}$}&
\resizebox{4mm}{!}{$\scriptstyle{i+1}$}&
\resizebox{4mm}{!}{$\scriptstyle{i+3}$}&
\dots&
\scriptstyle{n}\\
\scriptstyle{i}\\
\resizebox{4mm}{!}{$\scriptstyle{i+2}$}
\end{ytableau}
\end{displaymath}
Consequently, we also have $\mathbf{a}(x_{n,i})=\mathbf{a}(y_{n,i})=3$.
Clearly, $x_{n,i}\neq y_{n,i}$. Our aim is to show that 
\begin{equation}\label{eq-case3-1}
\theta_{x_{n,i}} L_{\sigma_{n,i}}\cong \theta_{y_{n,i}} L_{\sigma_{n,i}}\neq 0, 
\end{equation}
which would imply that $\sigma_{n,i}$ is Kostant negative
by \cite[Proposition~8.14]{KMM}.

\begin{lemma}\label{s5.6-lem1}
We have:
\begin{enumerate}[$($a$)$]
\item\label{s5.6-lem1-1}
$\underline{H}_{s_1s_2\dots s_{i+1}}=
\underline{H}_{s_1}\underline{H}_{s_2}\dots
\underline{H}_{s_{i+1}}$,
\item\label{s5.6-lem1-2}
$\underline{H}_{x_{n,i}}+\underline{H}_{s_1s_2\dots s_{i-2}s_{i}s_{i+1}}=
\underline{H}_{s_i}\underline{H}_{s_1}\underline{H}_{s_2}\dots
\underline{H}_{s_{i+1}}$,
\item\label{s5.6-lem1-3}
$\underline{H}_{s_{n-1}s_{n-2}\dots s_{i-1}}=
\underline{H}_{s_{n-1}}\underline{H}_{s_{n-2}}\dots
\underline{H}_{s_{i-1}}$,
\item\label{s5.6-lem1-4}
$\underline{H}_{y_{n,i}}+\underline{H}_{s_{n-1}s_{n-2}\dots s_{i-2}s_{i}s_{i-1}}=
\underline{H}_{s_{n-1}}\underline{H}_{s_{n-2}}\dots
\underline{H}_{s_{i-1}}$.
\end{enumerate}
\end{lemma}

\begin{proof}
Claims \eqref{s5.6-lem1-1} and
\eqref{s5.6-lem1-3} follow from 
Lemma~\ref{lem-disjointsupports} 
(alternatively, from \cite[Theorem~1]{BW}.

Let us prove Claim~\eqref{s5.6-lem1-2}, Claim~\eqref{s5.6-lem1-4}
is proved similarly. 
Using commutativity, we have
\begin{displaymath}
\underline{H}_{s_i}\underline{H}_{s_1}
\underline{H}_{s_2}\dots \underline{H}_{s_{i+1}}=
\underline{H}_{s_1}\dots \underline{H}_{s_{i-2}}
\underline{H}_{s_i}\underline{H}_{s_{i-1}}
\underline{H}_{s_i}\underline{H}_{s_{i+1}}.
\end{displaymath}
Next we note that 
\begin{displaymath}
\underline{H}_{s_i}\underline{H}_{s_{i-1}}
\underline{H}_{s_i}=\underline{H}_{s_i}
+\underline{H}_{s_is_{i-1}s_i}.
\end{displaymath}
Finally,
\begin{displaymath}
\underline{H}_{s_1}\underline{H}_{s_2}\dots
\underline{H}_{s_{i-2}}\underline{H}_{s_i}=
\underline{H}_{s_1s_2\dots s_{i-2}s_{i}}
\end{displaymath}
as well as 
\begin{displaymath}
\underline{H}_{s_1}\underline{H}_{s_2}\dots
\underline{H}_{s_{i-2}}\underline{H}_{s_is_{i-1}s_i}
=\underline{H}_{x_{n,i}}
\end{displaymath}
by Lemma~\ref{lem-disjointsupports}.
This implies Claim~\eqref{s5.6-lem1-2}.
\end{proof}

\begin{lemma}\label{s5.6-lem2}
We have both $\theta_{x_{n,i}} L_{\sigma_{n,i}}\neq 0$ and 
$\theta_{y_{n,i}} L_{\sigma_{n,i}}\neq 0$.
\end{lemma}

\begin{proof}
We prove that $\theta_{x_{n,i}} L_{\sigma_{n,i}}\neq 0$
(or, equivalently, 
$\underline{\hat{H}}_{\sigma_{n,i}}\underline{H}_{x_{n,i}}\neq 0$), 
the second claim follows from the symmetry of the root system.

Taking Lemma~\ref{s5.6-lem1} into account, our strategy to 
prove Lemma~\ref{s5.6-lem2} will be as follows:
\begin{itemize}
\item we first observe that 
$A:=\underline{\hat{H}}_{\sigma_{n,i}}\underline{H}_{s_i}\underline{H}_{s_1}\underline{H}_{s_2}\dots
\underline{H}_{s_{i+1}}$ is non-zero;
\item then we show that 
$\underline{\hat{H}}_{\sigma_{n,i}}\underline{H}_{s_1}
\underline{H}_{s_2}\dots
\underline{H}_{s_{i-2}}\underline{H}_{s_i}\underline{H}_{s_{i+1}}$
does not have any non-zero components
(with respect to the dual KL basis) in degrees $\pm 3$;
\item and, finally, 
we show that $A$ has non-zero components
(with respect to the dual KL basis) in degree $3$.
\end{itemize}

We start by noting that 
$\underline{\hat{H}}_{\sigma_{n,i}}\underline{H}_{s_i}$ is non-zero
since $s_i$ is in the right descent of $\sigma_{n,i}$ and,
consequently, $\underline{\hat{H}}_{\sigma_{n,i}}\underline{H}_{s_i}$
has $\underline{\hat{H}}_{\sigma_{n,i}}$ in degrees $\pm 1$
by \eqref{eq:kl-mu4}. Similarly, 
$\underline{\hat{H}}_{\sigma_{n,i}}\underline{H}_{s_1}$ is non-zero.
Since $s_1\sim_R s_1s_2\dots s_{i+1}$, it follows that 
$\underline{\hat{H}}_{\sigma_{n,i}}\underline{H}_{s_1}
\underline{H}_{s_2}\dots \underline{H}_{s_{i+1}}$
is non-zero. This implies that $A\neq 0$ 
and our first step is complete.

Let us take a closer look at the elements  
$\underline{\hat{H}}_{\sigma_{n,i}}\underline{H}_{s_1}
\underline{H}_{s_2}\dots \underline{H}_{s_r}$, where
$1\leq r\leq i-2$. Note that the corresponding modules
in $\mathcal{O}_0$ are indecomposable
by \cite[Proposition~2]{CMZ} since 
$\theta_{s_1}L_{\sigma_{n,i}}$ is indecomposable. 
The module $\theta_{s_1}L_{\sigma_{n,i}}$ has simple
top and socle isomorphic to $L_{\sigma_{n,i}}$. Both are killed
by $\theta_{s_2}$ as $\sigma_{n,i}s_2>\sigma_{n,i}$
(assuming $2\leq i-2$). Therefore the Jantzen middle of
$\theta_{s_1}L_{\sigma_{n,i}}$ contains a unique simple
which is not killed by $\theta_{s_2}$, namely 
$L_{\sigma_{n,i}s_1}$. This means that 
$\underline{\hat{H}}_{\sigma_{n,i}}\underline{H}_{s_1}
\underline{H}_{s_2}$ equals 
$\underline{\hat{H}}_{\sigma_{n,i}s_1}\underline{H}_{s_2}$.
Proceeding inductively, we obtain that
\begin{displaymath}
\underline{\hat{H}}_{\sigma_{n,i}}\underline{H}_{s_1}
\underline{H}_{s_2}\dots \underline{H}_{s_{i-2}}
=\underline{\hat{H}}_{\sigma_{n,i}s_1s_2\dots s_{i-3}}
\underline{H}_{s_{i-2}}
\end{displaymath}
and the latter has a copy of 
$\underline{\hat{H}}_{\sigma_{n,i}s_1s_2\dots s_{i-3}}$
in degrees $\pm 1$ and is zero (when written with respect to the
dual KL basis) in all degrees outside $\{-1,0,1\}$.
Note also that 
$\underline{\hat{H}}_{\sigma_{n,i}s_1s_2\dots s_{i-3}}
\underline{H}_{s_{i+2}}=0$ as
$\sigma_{n,i}s_1s_2\dots s_{i-3}s_{i+2}>\sigma_{n,i}s_1s_2\dots s_{i-3}$.
This implies that the element 
$\underline{\hat{H}}_{\sigma_{n,i}}\underline{H}_{s_1}
\underline{H}_{s_2}\dots
\underline{H}_{s_{i-2}}\underline{H}_{s_i}\underline{H}_{s_{i+1}}$
does not have any non-zero components
(with respect to the dual KL basis) in degrees $\pm 3$,
which gives us our second step.

For the last step, we note that 
$\underline{\hat{H}}_{\sigma_{n,i}s_1s_2\dots s_{i-3}}\underline{H}_{s_{i-1}}=0$
since
\begin{displaymath}
\sigma_{n,i}s_1s_2\dots s_{i-3}s_{i-1}>\sigma_{n,i}s_1s_2\dots s_{i-3} 
\end{displaymath}
and hence 
\begin{displaymath}
\underline{\hat{H}}_{\sigma_{n,i}s_1s_2\dots s_{i-3}}
\underline{H}_{s_{i-2}}\underline{H}_{s_{i-1}}=
\underline{\hat{H}}_{\sigma_{n,i}s_1s_2\dots s_{i-3}s_{i-2}}\underline{H}_{s_{i-1}}
\end{displaymath}
which has a copy of $\underline{\hat{H}}_{\sigma_{n,i}s_1s_2\dots s_{i-3}s_{i-2}}$
in degree $1$. But now we have
\begin{displaymath}
\sigma_{n,i}s_1s_2\dots s_{i-3}s_{i-2}s_{i}<
\sigma_{n,i}s_1s_2\dots s_{i-3}s_{i-2} 
\end{displaymath}
and hence
$\underline{\hat{H}}_{\sigma_{n,i}s_1s_2\dots s_{i-3}s_{i-2}}
\underline{H}_{s_{i-1}}\underline{H}_{s_{i}}$ has 
$\underline{\hat{H}}_{\sigma_{n,i}s_1s_2\dots s_{i-3}s_{i-2}}\underline{H}_{s_{i}}$
as a summand centered in degree $1$. Note that 
$\underline{\hat{H}}_{\sigma_{n,i}s_1s_2\dots s_{i-3}s_{i-2}}\underline{H}_{s_{i}}\neq 0$ implies 
\begin{displaymath}
\underline{\hat{H}}_{\sigma_{n,i}s_1s_2\dots s_{i-3}s_{i-2}}\underline{H}_{s_{i}}\underline{H}_{s_{i+1}}\neq 0 
\end{displaymath}
as $\underline{H}_{s_{i}}\underline{H}_{s_{i+1}}=
\underline{H}_{s_{i}s_{i+1}}$ and $s_{i}\sim_R s_{i}s_{i+1}$.
Moreover, as the $\mathbf{a}$-value of $s_{i}s_{i+1}$ is $1$,
the element $\underline{\hat{H}}_{\sigma_{n,i}s_1s_2\dots s_{i-3}s_{i-2}}\underline{H}_{s_{i}}\underline{H}_{s_{i+1}}$ must be non-zero in degrees
$\{-1,0,1\}$. This implies that 
$\underline{\hat{H}}_{\sigma_{n,i}s_1s_2\dots s_{i-3}s_{i-2}}
\underline{H}_{s_{i-1}}\underline{H}_{s_{i}}\underline{H}_{s_{i+1}}$
must be non-zero in degree $2$. Consequently, 
$A$ must be non-zero in degree $3$. This completes the proof.
\end{proof}

\begin{lemma}\label{s5.6-lem3}
Both $\theta_{x_{n,i}} L_{\sigma_{n,i}}$ and 
$\theta_{y_{n,i}} L_{\sigma_{n,i}}$ are indecomposable modules.
\end{lemma}

\begin{proof}
We prove that  $\theta_{x_{n,i}} L_{\sigma_{n,i}}$ is indecomposable,
for $\theta_{y_{n,i}} L_{\sigma_{n,i}}$ the proof is similar.
From the proof of Lemma~\ref{s5.6-lem2}, it follows that 
\begin{displaymath}
\theta_{x_{n,i}} L_{\sigma_{n,i}}\cong
\theta_{s_is_{i-2}s_{i-1}s_is_{i+1}}
L_{\sigma_{n,i}s_1s_2\dots s_{i-3}}.
\end{displaymath}
The element $s_is_{i-2}s_{i-1}s_is_{i+1}$
has support $\{i-2,i-1,i,i+1,i+2\}$ of length $5$.
Therefore the claim of our lemma follows from
\cite[Corollary~2.8]{CM2}.
\end{proof}

\begin{lemma}\label{s5.6-lem4}
If $\theta_{x_{6,3}} L_{\sigma_{6,3}}\cong\theta_{y_{6,3}} L_{\sigma_{6,3}}$,
then $\theta_{x_{n,i}} L_{\sigma_{n,i}}\cong\theta_{y_{n,i}} L_{\sigma_{n,i}}$,
for all $n\geq 6$ and $i\in\{3,4,\dots,n-3\}$.
\end{lemma}

\begin{proof}
The general case $\theta_{x_{n,i}} L_{\sigma_{n,i}}\cong\theta_{y_{n,i}} L_{\sigma_{n,i}}$
can be inductively reduced to the case $n=6$
and $i=3$ using the following argument. Assuming $i>3$, 
in the proof of Lemma~\ref{s5.6-lem2} we saw that 
\begin{displaymath}
\theta_{x_{n,i}}  L_{\sigma_{n,i}}\cong
\theta_{s_1x_{n,i}}L_{\sigma_{n,i}s_1}.
\end{displaymath}
The consecutive $n-1$-pattern in $\sigma_{n,i}s_1$ starting at 
position $2$ is given by $\sigma_{n-1,i-1}$. The element $s_1x_{n,i}$
is obtained from $x_{n,i}$ by deleting the only occurrence of $s_1$
and the outcome $s_1x_{n,i}$, considered as the element of the parabolic
subgroup of $S_n$ obtained by deleting $s_1$, is exactly the
corresponding element $x_{n-1,i-1}$. The element $y_{n,i}$, considered as 
the element of the same parabolic subgroup, is exactly the corresponding
$y_{n-1,i-1}$. Now the proof of Theorem~\ref{thm-main}, based on the
equivalence in \cite[Theorem~37]{CMZ}, gives the 
reduction step from $(n,i)$ to $(n-1,i-1)$ provided that $i>3$
in the sense that existence of an isomorphism 
$\theta_{x_{n-1,i-1}} L_{\sigma_{n-1,i-1}}\cong
\theta_{y_{n-1,i-1}} L_{\sigma_{n-1,i-1}}$ implies existence of an
isomorphism 
$\theta_{x_{n,i}} L_{\sigma_{n,i}}\cong\theta_{y_{n,i}} L_{\sigma_{n,i}}$,

If $i<n-3$, we  can do a similar manipulation 
removing $s_{n-1}$ which  shows that existence of an isomorphism 
$\theta_{x_{n-1,i}} L_{\sigma_{n-1,i}}\cong
\theta_{y_{n-1,i}} L_{\sigma_{n-1,i}}$ implies existence of an
isomorphism 
$\theta_{x_{n,i}} L_{\sigma_{n,i}}\cong\theta_{y_{n,i}} L_{\sigma_{n,i}}$.
After a finite number of such manipulations, we eventually
end up exactly in the case
$n=6$ and $i=3$. The claim of the lemma follows.
\end{proof}

To complete the proof of Theorem~\ref{thm-cusp3}, it remains to show that
$\theta_{x_{6,3}} L_{\sigma_{6,3}}\cong\theta_{y_{6,3}} L_{\sigma_{6,3}}$.
This will require some effort and preparation.

\begin{lemma}\label{s5.6-lem5}
We have 
$\underline{\hat{H}}_{\sigma_{6,3}}
\underline{{H}}_{x_{6,3}}=
\underline{\hat{H}}_{\sigma_{6,3}}
\underline{{H}}_{y_{6,3}}
$.
\end{lemma}

\begin{proof}
A straightforward  GAP3 computation gives that both sides are equal to
\begin{multline*}
\underline{\hat{H}}_{32145}+
\underline{\hat{H}}_{1232145}+
\underline{\hat{H}}_{1321435}+
\underline{\hat{H}}_{3214543}+
\underline{\hat{H}}_{3454321}+
\underline{\hat{H}}_{123214543}+
\underline{\hat{H}}_{123454321}+\\+
\underline{\hat{H}}_{132454321}+
(v+v^{-1})\underline{\hat{H}}_{132145}+
(v+v^{-1})\underline{\hat{H}}_{321454}+
(v+v^{-1})\underline{\hat{H}}_{12321454}+\\+
(v+v^{-1})\underline{\hat{H}}_{13214354}+
(2v+2v^{-1})\underline{\hat{H}}_{13214543}+
(v+v^{-1})\underline{\hat{H}}_{13454321}+\\+
(v+v^{-1})\underline{\hat{H}}_{32145432}+
(v+v^{-1})\underline{\hat{H}}_{1232145432}+
(v+v^{-1})\underline{\hat{H}}_{1324354321}+\\+
(v^2+2+v^{-2})\underline{\hat{H}}_{1321454}+
(v^2+2+v^{-2})\underline{\hat{H}}_{132143543}+\\+
(v^2+2+v^{-2})\underline{\hat{H}}_{132145432}+
(v^3+3v+3v^{-1}+v^{-3})\underline{\hat{H}}_{1321435432}
\end{multline*}
(here all indices are given in terms of reduced expressions).
\end{proof}

\begin{lemma}\label{s5.6-lem6}
We have:
\begin{enumerate}[$($a$)$]
\item\label{s5.6-lem6-1}  $\underline{{H}}_{x_{6,3}}
\underline{{H}}_{y_{6,3}^{-1}}=
(v^3+3v+3v^{-1}+v^{-3})\underline{{H}}_{1232435}$
\item\label{s5.6-lem6-2}  
$\underline{\hat{H}}_{\sigma_{6,3}}\underline{{H}}_{1232435}=
(v^3+3v+3v^{-1}+v^{-3})\underline{\hat{H}}_{\sigma_{6,3}}+u$,
where $u$ is a linear combination of elements in the dual KL basis
with coefficients in $\mathbb{Z}v^2+\mathbb{Z}v+\mathbb{Z}+
\mathbb{Z}v^{-1}+\mathbb{Z}v^{-2}$.
\item\label{s5.6-lem6-3}  $\underline{{H}}_{x_{6,3}}
\underline{{H}}_{x_{6,3}^{-1}}=
(v^3+3v+3v^{-1}+v^{-3})\underline{{H}}_{12321}
+ (v^2+2+v^{-2})\underline{{H}}_{12324321}$.
\item\label{s5.6-lem6-4}  
$\underline{\hat{H}}_{\sigma_{6,3}}\underline{{H}}_{12324321}=0 $.
\item\label{s5.6-lem6-5}  
$\underline{\hat{H}}_{\sigma_{6,3}}\underline{{H}}_{12321}=
(v^3+3v+3v^{-1}+v^{-3})\underline{\hat{H}}_{\sigma_{6,3}}+v$,
where $v$ is a linear combination of elements in the dual KL basis
with coefficients in $\mathbb{Z}v^2+\mathbb{Z}v+\mathbb{Z}+
\mathbb{Z}v^{-1}+\mathbb{Z}v^{-2}$.
\end{enumerate}
\end{lemma}

\begin{proof}
This is checked by a GAP3 computation.
\end{proof}

Below, as usual, we use $\mathrm{hom}$ to denote homomorphisms
in ${}^\mathbb{Z}\mathcal{O}_0$.

\begin{lemma}\label{s5.6-lem7}
We have both 
\begin{displaymath}
\mathrm{hom}(\theta_{x_{6,3}} L_{\sigma_{6,3}},
\theta_{y_{6,3}} L_{\sigma_{6,3}})\neq 0\quad
\text{ and } \quad\mathrm{hom}(\theta_{y_{6,3}} L_{\sigma_{6,3}},
\theta_{x_{6,3}} L_{\sigma_{6,3}})\neq 0 .¨
\end{displaymath}
\end{lemma}

\begin{proof}
We prove the first claim, the second one is similar.
By adjunction, we have 
\begin{equation}\label{eq-s5.6-lem7-1}
\mathrm{hom}(\theta_{x_{6,3}} L_{\sigma_{6,3}},
\theta_{y_{6,3}} L_{\sigma_{6,3}})\cong 
\mathrm{hom}(\theta_{y_{6,3}^{-1}}\circ\theta_{x_{6,3}} L_{\sigma_{6,3}},
L_{\sigma_{6,3}}).
\end{equation}
By Lemma~\ref{s5.6-lem6}\eqref{s5.6-lem6-1}, we have 
\begin{displaymath}
\theta_{y_{6,3}^{-1}}\circ\theta_{x_{6,3}}\cong
\theta_{1232435}\langle 3\rangle\oplus
\theta_{1232435}\langle 1\rangle^{\oplus 2}\oplus
\theta_{1232435}\langle -1\rangle^{\oplus 2}\oplus
\theta_{1232435}\langle -3\rangle.
\end{displaymath}
By Lemma~\ref{s5.6-lem6}\eqref{s5.6-lem6-2},
the module $\theta_{1232435}L_{\sigma_{6,3}}$
has $L_{\sigma_{6,3}}\langle 3\rangle$ in its top.
Therefore the right hand side of \eqref{eq-s5.6-lem7-1}
is non-zero, implying the claim of the lemma.
\end{proof}

\begin{lemma}\label{s5.6-lem8}
For $i\in\mathbb{Z}$, the inequality
$\mathrm{hom}(\theta_{12321} L_{\sigma_{6,3}}\langle i\rangle,
L_{\sigma_{6,3}})\neq 0$ implies $i=-3$.
\end{lemma}

\begin{proof}
By GAP3 computation, we have 
$\underline{H}_{12321}=
\underline{H}_{1}\underline{H}_{232}\underline{H}_{1}$
and, consequently, we have
$\theta_{1}\circ\theta_{232}\circ\theta_{1}\cong \theta_{12321}$.

Further, by GAP3 computation, the only simple subquotient of
$\theta_1 L_{\sigma_{6,3}}$ that is not killed by 
$\theta_{232}$ is $L_{132145432}$ (it only appears in degree 
zero where it has multiplicity $1$). This means that 
$\theta_{232}\circ\theta_{1}L_{\sigma_{6,3}}\cong \theta_{232}L_{132145432}$.

Since $232$ is the longest elements of a parabolic subgroup
and has $\mathbf{a}$-value $3$, 
the module $\theta_{232}L_{132145432}$ has simple top,
namely, $L_{132145432}\langle 3\rangle$. By adjunction, we have
\begin{displaymath}
\mathrm{hom}(\theta_{12321} L_{\sigma_{6,3}}\langle i\rangle,
L_{\sigma_{6,3}})\cong
\mathrm{hom}(\theta_{232} L_{132145432}\langle i\rangle,
\theta_{1}L_{\sigma_{6,3}}).
\end{displaymath}
Since $\theta_{232} L_{132145432}$ has simple top in degree $3$
and $L_{132145432}$ appears with multiplicity $1$ in 
$\theta_{1}L_{\sigma_{6,3}}$ in degree $0$, it follows that $i=-3$.
\end{proof}

\begin{lemma}\label{s5.6-lem9}
For $i\in\mathbb{Z}$, we have:
\begin{displaymath}
\mathrm{hom}(\theta_{x_{6,3}} L_{\sigma_{6,3}},
\theta_{x_{6,3}} L_{\sigma_{6,3}}\langle i\rangle)=
\begin{cases}
\mathbb{C},& i=0,6;\\
\mathbb{C}^2,& i=2,4;\\
0,& \text{else}.
\end{cases}
\end{displaymath}
\end{lemma}

\begin{proof}
By adjunction,
\begin{displaymath}
\mathrm{hom}(\theta_{x_{6,3}} L_{\sigma_{6,3}},
\theta_{x_{6,3}} L_{\sigma_{6,3}}\langle i\rangle)=
\mathrm{hom}(\theta_{x_{6,3}^{-1}} \circ\theta_{x_{6,3}} L_{\sigma_{6,3}},
L_{\sigma_{6,3}}\langle i\rangle).
\end{displaymath}
By Lemma~\ref{s5.6-lem6}\eqref{s5.6-lem6-3}, we have 
\begin{multline*}
\theta_{x_{6,3}^{-1}}\circ\theta_{x_{6,3}}\cong
\theta_{12321}\langle 3\rangle\oplus
\theta_{12321}\langle 1\rangle^{\oplus 2}\oplus
\theta_{12321}\langle -1\rangle^{\oplus 2}\oplus
\theta_{12321}\langle -3\rangle\oplus \\ \oplus 
\theta_{12324321}\langle 2\rangle\oplus
\theta_{12324321}^{\oplus 2}\oplus
\theta_{12324321}\langle -2\rangle.
\end{multline*}
By Lemma~\ref{s5.6-lem6}\eqref{s5.6-lem6-4},
we have $\theta_{12324321}L_{\sigma_{6,3}}=0$
and hence we can ignore the corresponding summands.
Now the claim of the lemma follows from 
Lemma~\ref{s5.6-lem6}\eqref{s5.6-lem6-5}
and Lemma~\ref{s5.6-lem8}.
\end{proof}

\begin{lemma}\label{s5.6-lem91}
Both $\theta_{x_{6,3}}L_{\sigma_{6,3}}$ and 
$\theta_{y_{6,3}}L_{\sigma_{6,3}}$ have simple tops and socles.
\end{lemma}

\begin{proof}
We prove the claim for $\theta_{x_{6,3}}L_{\sigma_{6,3}}$.
The claim for  $\theta_{y_{6,3}}L_{\sigma_{6,3}}$ then follows
from the symmetry of the root system.

As $\theta_{x_{6,3}}L_{\sigma_{6,3}}$ is self-dual, it is enough
to prove the claim for the socle. All simple subquotients of 
$\theta_{x_{6,3}}L_{\sigma_{6,3}}$ are given in the proof of
Lemma~\ref{s5.6-lem5}. From Lemma~\ref{s5.6-lem9} we know that
the endomorphism algebra of $\theta_{x_{6,3}}L_{\sigma_{6,3}}$
is positively graded, in particular, the degree zero part
is just $\mathbb{C}$ and hence consists of  the scalar
multiples of the identity. There is a socle constituent in
the extreme degree $3$ of the module $\theta_{x_{6,3}}L_{\sigma_{6,3}}$
and that degree has total length $1$. If there
is any other socle constituent, it must live in strictly
positive degrees, so in either degree $1$ or in degree $2$.
Also, any simple socle constituent must be from the same
KL right cell as $\sigma_{6,3}$. Here is the list of all 
simple subquotients of $\theta_{x_{6,3}}L_{\sigma_{6,3}}$
which appear in positive degrees and come from the same
KL right cells:
\begin{displaymath}
L_{1321454}, \quad
L_{13214543},\quad
L_{132143543},\quad
L_{132145432},\quad
L_{1321435432}.
\end{displaymath}

Next we note that the right descent set of  $x_{6,3}$ consists of $2$
and $4$. Hence 
\begin{displaymath}
\theta_2\theta_{x_{6,3}}=\theta_{x_{6,3}}
\langle 1\rangle\oplus\theta_{x_{6,3}}\langle -1\rangle
\text{ and }
\theta_4\theta_{x_{6,3}}=\theta_{x_{6,3}}
\langle 1\rangle\oplus\theta_{x_{6,3}}\langle -1\rangle 
\end{displaymath}
and therefore, by adjunction,
$\mathrm{hom}(L_w,\theta_{x_{6,3}}L_{\sigma_{6,3}})\neq 0$
implies both $\theta_2 L_w\neq 0$ and $\theta_4 L_w\neq 0$.
In other words, both $2$ and $4$ must belong to the right descent set
of $w$. In the above list, only two elements contain $2$ and $4$
in their right descent set, these are 
\begin{displaymath}
L_{1321454}\text{ and }
L_{1321435432}.
\end{displaymath}
However, a GAP3 computation shows that 
$\theta_{x_{6,3}^{-1}}L_{1321454}=0$ which, by adjunction,
implies that $\mathrm{hom}(L_{1321454},\theta_{x_{6,3}}L_{\sigma_{6,3}})=0$.
This leaves us only with $L_{1321435432}$, which appears in 
degree $3$ (a top constituent), in degree $-3$ (a socle constituent)
and also with multiplicity $2$ in degrees $\pm 1$. So, the only 
alternative to $\theta_{x_{6,3}}L_{\sigma_{6,3}}$ having simple top
and socle is that it could have a socle constituent
in degree $1$ and, respectively, a top constituent in degree $-1$.
So, assume
\begin{displaymath}
\mathrm{hom}(L_{1321435432}\langle -1\rangle,
\theta_{x_{6,3}}L_{\sigma_{6,3}})\neq 0. 
\end{displaymath}

As $\underline{H}_{x_{6,3}}=\underline{H}_1\underline{H}_{232}\underline{H}_4$,
we have $\theta_{x_{6,3}}=\theta_{4}\circ\theta_{232}\circ\theta_{1}$ and hence,
by adjunction,
\begin{displaymath}
\mathrm{hom}(\theta_4 L_{1321435432}\langle -1\rangle,
\theta_{232}\circ \theta_{1} L_{\sigma_{6,3}})\neq 0. 
\end{displaymath}
The only simple subquotient of  $\theta_{1} L_{\sigma_{6,3}}$
that survives $\theta_{232}$ is $L_{132145432}$. Hence
\begin{displaymath}
\theta_{232}\circ \theta_{1} L_{\sigma_{6,3}}\cong 
\theta_{232}L_{132145432}.
\end{displaymath}
This gives
\begin{displaymath}
\mathrm{hom}(\theta_4 L_{1321435432}\langle -1\rangle,
\theta_{232}L_{132145432})\neq 0. 
\end{displaymath}
This is, however, not possible as the module 
$\theta_4 L_{1321435432}\langle -1\rangle$ lives in
degrees $0,1,2$ while the module 
$\theta_{232}L_{132145432}$ has simple socle in degree $3$
(since $232$ has $\mathbf{a}$-value $3$ and is the longest element
of a parabolic subalgebra). The obtained contradiction completes the proof.
\end{proof}

Now let $\varphi: \theta_{x_{6,3}}L_{\sigma_{6,3}}
\to \theta_{y_{6,3}}L_{\sigma_{6,3}}$ be a
non-zero morphisms given by Lemma~\ref{s5.6-lem7}.
Since $\theta_{y_{6,3}}L_{\sigma_{6,3}}$ has simple
socle by Lemma~\ref{s5.6-lem91}, this socle must be 
contained in the image of  $\varphi$. Since $\varphi$ has degree zero,
this implies that $\varphi$ does not annihilate 
the simple socle of $\theta_{x_{6,3}}L_{\sigma_{6,3}}$,
in particular, $\varphi$ is injective. 
Since the characters of
$\theta_{x_{6,3}}L_{\sigma_{6,3}}$ and 
$\theta_{y_{6,3}}L_{\sigma_{6,3}}$ coincide by Lemma~\ref{s5.6-lem5},
it follows that $\varphi$ is an  isomorphism.

This establishes the Kostant negativity of $L_{\sigma_{6,3}}$
and thus our proof of Theorem~\ref{thm-cusp3} is now complete.

\vspace{2mm}

\noindent
Department of Mathematics, Uppsala University, Box. 480,
SE-75106, Uppsala, \\ SWEDEN, 
emails:
{\tt samuel.creedon\symbol{64}math.uu.se}\hspace{5mm}
{\tt mazor\symbol{64}math.uu.se}

\end{document}